\documentclass[11pt]{amsart}

\usepackage{amsmath,amsfonts,amsthm, amssymb,color}

\newtheorem{theorem}{Theorem} [section]

\newtheorem{proposition}[theorem]{Proposition}
\newtheorem{corollary}[theorem]{Corollary}

\newtheorem{example}[theorem]{Example}
\newtheorem{question}[theorem]{Question}

\theoremstyle{definition}

\newcommand\N{{\mathbb N}}
\newcommand\Z{{\mathbb Z}}
\newcommand\C{{\mathbb C}}

\newcommand\slex{<_{sl}}
\newcommand\conj{\sim_c}

\newcommand\geol{{\Gamma}}   
\newcommand\wgeol{{geodesic language}}   
\newcommand\geos{{\gamma}}    
\newcommand\wgeos{{geodesic growth series}}    

\newcommand\sphl{{\Sigma}}    
\newcommand\wsphl{{spherical language}}    
\newcommand\sphs{{\sigma}}    
\newcommand\wsphs{{spherical growth series}}    

\newcommand\geocl{{\widetilde {\Gamma}}}   
\newcommand\wgeocl{{geodesic conjugacy language}}   
\newcommand\geocs{{\widetilde {\gamma}}}    
\newcommand\wgeocs{{geodesic conjugacy growth series}}  
  
\newcommand\sphcl{{\widetilde {\Sigma}}}  
\newcommand\wsphcl{{spherical conjugacy language}}    
\newcommand\sphcs{{\widetilde {\sigma}}}    
\newcommand\wsphcs{{spherical conjugacy growth series}}    

\newcommand\eqcl{{\widetilde {\mathcal E}}}    
\newcommand\weqcl{{equality conjugacy language}}    
\newcommand\eqcs{{\widetilde {\epsilon}}}    
\newcommand\weqcs{{equality conjugacy growth series}}    

\newcommand\conjl{{conjugacy language}}

\newcommand\sphr{{spherical}}

\newcommand\geocon{{conjugacy geodesic}}  

\newcommand\ra{\rightarrow}
\newcommand\sr{\stackrel}
\newcommand\lra{\longrightarrow}

\newcommand\ct{{conjugationally trimmed}}
\newcommand\tri{{trimmed}}

\newcommand\la{\lambda}
\newcommand\lm{\Lambda}
\newcommand\ei{=_{G_i}}
\newcommand\tht{\Theta}

\newcommand\aaa{\{a_1\}}
\newcommand\bb{\{a_2\}}
\newcommand\cc{\{a_3\}}
\newcommand\ab{\{a_1, a_2\}}
\newcommand\ac{\{a_1, a_3\}}

\newcommand{\uu}[1]{{\underline{\color{blue}#1}}}

\begin{document}

\author{L. Ciobanu and S. Hermiller}
\title{Conjugacy growth series and languages in groups}
\date{}
\begin{abstract}
In this paper we introduce the \wgeocl\ and \wgeocs\ 
for a finitely generated group.  We study the effects of
various group constructions on
rationality of
both the \wgeocs\ and \wsphcs, as well as on regularity of
the \wgeocl\ and \wsphcl.
In particular, we show that regularity of the \wgeocl\ is
preserved by the graph product construction, and rationality
of the \wgeocs\ is preserved by both direct and free products.

\bigskip


\end{abstract}
\maketitle

\section{Introduction}\label{sec:intro}

For a finitely generated group, 
several growth functions and series associated with
elements or conjugacy classes of the group have
been studied.  In this paper, we study a
conjugacy growth series first examined by Rivin~\cite{rivin},
and introduce a new
growth series arising from the conjugacy classes, which
we show admits much stronger closure properties.  
We also study the regularity properties of languages
associated to the set of conjugacy classes.

Let $G=\langle X \rangle$ be a group generated by a finite inverse-closed
generating set $X$.
For each word $w \in X^*$, let $l(w)=l_X(w)$ denote the
length of this word over $X$.
Any language $L \subseteq X^*$ over the finite alphabet $X$ 
gives rise to two growth functions $:\N \cup \{0\} \ra \N \cup \{0\}$, 
namely the ``usual'' or {\em cumulative growth function} $\beta_L$ defined
by $\beta_L(n) := |\{w \in L \mid l(w) \le n\}|$ and the 
{\em strict growth function}  $\phi_L(n) := |\{w \in L \mid l(w) = n\}|$.
The generating series associated to these functions
are the ``usual'' or {\em cumulative growth series} 
$b_L(z) := \sum_{i=0}^\infty \beta_L(i)z^i$ and
the {\em strict growth series} $f_L(z) := \sum_{i=0}^\infty \phi_L(i)z^i$.
These growth functions and series are closely related, in that
$\phi_L(n)=\beta_L(n)-\beta_L(n-1)$ for all $n \ge 1$,
and $\phi_L(0)=\beta_L(0)$, and so these two series satisfy
the identity $f_L(z) = (1-z) b_L(z)$.
It is well known (see, for example,~\cite{brazil}) that
if the language $L$ is a regular language (i.e., the language
of a finite state automaton), then both of the series $b_L$ and $f_L$
are rational functions.
In this paper, we will focus on strict growth series associated
to two languages derived from the pair $(G,X)$.

Let $\conj$ denote the equivalence relation on $G$ given
by conjugacy, with set $G/\conj$ of equivalence classes,
and let $[g]_c$ denote the conjugacy class of $g \in G$.
Let $\pi:X^* \rightarrow G$ be the natural projection.
Fix a total ordering $<$ of the set $X$, and let $\slex$ be
the induced shortlex ordering of $X^*$.  
For each conjugacy class $c \in G/\conj$, there 
is a {\em shortlex conjugacy normal form} $z_c$ for $c$ 
that is the shortlex least word over $X$ representing an 
element of $G$ lying in $c$.  That is, 
$[\pi(z_c)]_c=c$, and for all 
$w \in X^*$ with $w \neq z_c$ and $[\pi(w)]_c=c$ 
we have $z_c \slex w$. 
We call the set 
\[\sphcl=\sphcl(G,X):=\{z_c \mid c \in G/\conj\}\]
the {\em \wsphcl} for $G$ over $X$.
Note that if $L$ is any other language of
length-minimal normal forms for the
conjugacy classes of $G$ over the generating set $X$,
then the growth functions (and hence also the
corresponding series) for the \wsphcl\ and $L$ must coincide;
that is, for all natural numbers $n$ we have
$\phi_\sphcl(n)=\phi_L(n)$.
The strict growth series 
\[\sphcs=\sphcs(G,X):=f_{\sphcl(G,X)}\]
is the {\em \wsphcs} for $G$ over $X$.
In Section~\ref{sec:spherical} we study this
language and series.
 
The corresponding cumulative growth function $\beta_{\sphcl(G,X)}$, 
known as the ``conjugacy growth function'', has been studied
by several authors; see the surveys
by Guba and Sapir~\cite{gubasapir} and
Breuillard and de Cornulier~\cite{bdc} for further
information on these functions.  For the class of
torsion-free word hyperbolic groups, 
Coornaert and Knieper~\cite{ck} have shown bounds on the
growth of the conjugacy growth function in terms of the exponential 
growth rate of the {\em \wsphl}
\[\sphl=\sphl(G,X):=\{y_g \mid g \in G\}\]
of shortlex normal forms for the elements of $G$;
here for each $g \in G$ the word $y_g \in X^*$ satisfies
$\pi(y_g)=g$ and whenever $w \in X^*$ with $w \neq y_g$ and $\pi(w)=g$ 
then $y_g \slex w$.
Analogous to the case of the conjugacy language above,
if $L$ is any other language of
geodesic normal forms for the
elements of $G$ over the generating set $X$,
then the growth functions for the \wsphl\ and $L$ must coincide;
that is, for all natural numbers $n$ we have
$\phi_\sphl(n)=\phi_L(n)$
 = the number
of elements of $G$ in the sphere of radius $n$
with respect to the word metric on $G$ induced by $X$.
The (usual) growth function of the group $G$ with respect to $X$, 
i.e.~the cumulative growth function $\beta_{\sphl(G,X)}$, 
is very well known and studied; see the texts by 
de la Harpe~\cite[Chapter~VI]{dlh} and Mann~\cite{mann} for 
surveys of results and open problems for the cumulative and strict
growth functions associated to this \wsphl.

Chiswell~\cite[Corollary 1]{chiswell} has shown that
rationality of the \wsphs\ (i.e.~the strict
growth series of the \wsphl\ $\sphl$) is preserved by the construction
of graph products of groups,
and in their proof of~\cite[Corollary~5.1]{hm}, Hermiller and
Meier show that regularity of the \wsphl\ is preserved by this
construction as well.  (The graph product construction includes
both direct and free products; see Section~\ref{sec:geodesic} for
the definition.)
In Proposition~\ref{prop:dirprodsph} we note that regularity
of the \wsphcl\ and rationality of the \wsphcs\ are 
preserved by direct products.
By contrast, in~\cite[Theorem~14.6]{rivin}, Rivin
has shown that although for the infinite cyclic group
$\Z=\langle a \rangle$ 
the \wsphcl\ is a regular language and
$\sphcs(\Z,\{a^{\pm 1}\})=\frac{1+z}{1-z}$ is rational,
the \wsphcs\ $\sphcs(F_2,\{a^{\pm 1},b^{\pm 1}\})$
for the free group on two generators $a,b$ is not a rational
function.  Combining this with the result above
on rationality of the growth of regular languages, 
this shows that the \wsphcl\ 
$\sphcl(F_2,\{a^{\pm 1},b^{\pm 1}\})$ is not regular.
In Section~\ref{sec:spherical}, we strengthen this
result to the broader class of context-free languages.

\medskip

\noindent{\bf Proposition~\ref{prop:freegpscl}.}
{\em Let $F=F(a,b)$ be a free group on generators $a, b$.
Then the \sphr\  conjugacy language $\sphcl(F, \{a^{\pm 1}, b^{\pm 1}\})$ 
is not context-free.
}

\medskip

Consequently neither regularity of the \wsphcl\ nor
rationality of the \wsphcs\ are preserved
by the free product construction.  
We also give further examples for which these properties
of {\wsphcl}s are not preserved by free products, in the following.

\medskip

\noindent{\bf Theorem~\ref{thm:freeprodfinite}.}
{\em For finite nontrivial groups $A$ and $B$ with generating sets $X_A=A\setminus 1_A$ 
and $X_B=B\setminus 1_B$, the free product group $A* B$ with generating set 
$X:=X_A \cup X_B$ has 
rational \wsphcs\ $\sphcs(A * B,X)$ 
if and only if $A=B=\mathbb{Z}/{2\mathbb{Z}}$.
Moreover, given any ordering of $X$ satisfying $a<b$ for all $a \in X_A$ and
$b \in X_B$, for the induced shortlex ordering
the associated \wsphcl\ $\sphcl(A * B,X)$ is regular 
if and only if $A=B=\mathbb{Z}/{2\mathbb{Z}}$.
} 

\medskip

The property of admitting a rational \wsphcs\ appears to be
extremely restrictive; indeed, Rivin~\cite[Conjecture~13.1]{rivin}
has conjectured that the only word hyperbolic groups that have a
rational \wsphcs\ are the virtually cyclic groups.
Theorem~\ref{thm:freeprodfinite} gives further evidence for 
this conjecture.

In \cite{HRR}, Holt, Rees, and R\"over also connect languages and conjugation 
in groups, but from a different perspective.
They consider the \textit{conjugacy problem} set associated a group $G$
with finite generating set $X$, 
i.e.~the set of ordered pairs $(u,v)$ such that $u$ and $v$ are 
conjugate in $G$.  They show that various notions of
context-freeness of this language can be used to
characterize the classes of 
virtually cyclic and virtually free groups. 

It is also natural to
consider all of the geodesics instead of a normal form set when
studying languages or growth series.
For each element $g$ of $G$, let $|g|(=|g|_X)$ denote
the length of a shortest representative word for $g$ over $X$.
A {\em geodesic}, then, is a word $w \in X^*$ with $l(w)=|\pi(w)|$.
The {\em length up to conjugacy} of $g$,
denoted by $|g|_c$, is defined to be 
\[|g|_c:=\min\{|h| \mid h \in [g]_c\}~,\]
and the length of the conjugacy
class is denoted by $|[g]_c|:=|g|_c$.
An element $w \in X^*$ satisfying $l(w)=|[\pi(w)]_c|$ 
will be called a {\em geodesic} of the pair $(G,X)$
{\em with respect to conjugacy}, or a {\em \geocon} word.
We define a new series 
 for the pair $(G,X)$ built from these words.  The set
\[ \geocl=\geocl(G,X):=\{w \in X^* \mid l(w)=|\pi(w)|_c\} \]
of {\geocon}s is the {\em \wgeocl}, and 
the strict growth series
\[ \geocs=\geocs(G,X) := f_{\geocl(G,X)} \]
is the {\em \wgeocs}, for $G$ over $X$.
In Section~\ref{sec:geodesic} we study this
language and series.

This \wgeocl\ is the canonical analog in the case of conjugacy
classes to the set
\[ \geol=\geol(G,X) := \{w \in X^* \mid l(w)=|\pi(w)|\} \]
of geodesic words, which we will call the {\em \wgeol} of
the group $G$ over $X$.  The corresponding strict growth
series will be called the {\em \wgeos}, and denoted
\[\geos = \geos(G,X) := f_{\geol(G,X)}.\]
See the paper by Grigorchuk and Nagnibeda~\cite[Section~6]{gn}
for a survey of results on this series.
For word hyperbolic groups, Cannon has shown 
that the \wgeol\ 
for every finite generating set is 
regular~\cite[Chapter~3]{echlpt}; that is, the group has
finitely many ``cone types''. 
 
Loeffler, Meier, and Worthington~\cite{lmw} have shown that
regularity of $\geol$ is
preserved by the graph product construction 
and rationality of $\geos$ is preserved by direct and free products. 
In Section~\ref{sec:geodesic} in Propositions~\ref{prop:geoword}
and~\ref{prop:geocon} we give several characterizations
of the geodesic words and \geocon\ words
for a graph product group in terms of the {\wgeol}s
and {\wgeocl}s for the vertex (factor) groups.
We use these to show that, unlike the
\sphr\ conjugacy case above, regularity of the 
pair of languages $\geocl$, $\geol$
is preserved when taking graph products.
 
\medskip

\noindent{\bf Theorem~\ref{thm:graphprodgcl}.}
{\em If $G$ is a graph product of groups $G_i$ with $1 \le i \le n$, and
each $G_i$ has a finite inverse-closed generating set $X_i$ such
that both $\geol(G_i,X_i)$ and $\geocl(G_i,X_i)$ are regular,
then both the \wgeol\ $\geol(G,\cup_{i=1}^n X_i)$ and
the \wgeocl\ 
$\geocl(G,\cup_{i=1}^n X_i)$ are also regular.
} 

\medskip

A corollary of Theorem~\ref{thm:graphprodgcl} is that for
every right-angled Artin group, right-angled Coxeter group,
and graph product of finite groups, with respect to
the ``standard'' generating set (that is, a union of the
generating sets of the vertex groups), the \wgeocl\ is regular 
and hence the
\wgeocs\ is rational. 

Rationality of {\wgeocs} is also preserved by
both free and direct products.  

\medskip

\noindent{\bf Theorem~\ref{thm:dirfreegcs}.}
{\em Let $G$ and $H$ be groups with finite inverse-closed generating
sets $A$ and $B$, respectively.  
Let $\geocs_G(z):=\geocs(G,A)(z)=\sum_{i=0}^\infty r_iz^i$ and 
$\geocs_H(z):=\geocs(H,B)(z)=\sum_{i=0}^\infty s_iz^i$ be the \wgeocs, 
and let $\geos_G:=\geos(G,A)$ and $\geos_H:=\geos(H,B)$
be the \wgeos\ for these pairs.

(i) The \wgeocs\ $\geocs_\times$ of the direct product 
$G \times H = \langle G,H \mid [G,H] \rangle$ of groups 
$G$ and $H$ with respect to the generating set $A \cup B$ is given by 
$\geocs_\times = \sum_{i=0}^\infty \delta_iz^i$
where $\delta_i:=\sum_{j=0}^i \binom{i}{j} r_j s_{i-j}$.

(ii) The \wgeocs\ $\geocs_*$ of the free product 
$G*H$ of the groups 
$G$ and $H$ with respect to the generating set $A \cup B$ is given by 
$$\geocs_* - 1 = (\geocs_G -1) + (\geocs_H - 1) -
  z \frac{d}{dz} \ln\left[ 1 - (\geos_G-1)(\geos_H-1) \right].$$
} 

\medskip

In our last result of Section~\ref{sec:geodesic},
we show in Proposition~\ref{prop:amalgprodfinite}
that geodesic conjugacy languages and series also
behave nicely for a free product with amalgamation of
finite groups 

\medskip

\noindent{\bf Proposition~\ref{prop:amalgprodfinite}.}
{\em If $G$ and $H$ are finite groups with a common subgroup $K$,
then the free product $G *_K H$ of $G$ and $H$ amalgamated over 
$K$, with respect 
to the generating set $X := G \cup H \cup K - \{1\}$,
has regular \wgeocl\ $\geocl(G *_K H,X)$
and rational \wgeocs\ $\geocs(G*_KH,X)$.}

\medskip

In Section~\ref{sec:open} we conclude with a few open questions.

\section{\sphr\  conjugacy series and languages}\label{sec:spherical}

In this section we collect information about the
closure properties of the class of pairs $(G,X)$
for which the \wsphcl\ is regular, or for which
the \wsphcs\ is rational. 

\begin{proposition}\label{prop:dirprodsph}
Let $G$ and $H$ be groups with finite inverse-closed
generating sets $X$ and $Y$, respectively. 
\begin{description}
\item[(i)] Let $<$ be a total ordering on $X \cup Y$ 
satisfying $x<y$ for all $x \in X$ and $y \in Y$;
we take all shortlex orderings to be defined from $<$
or its restriction to $X$ or $Y$.
If $\sphcl(G,X)$ and $\sphcl(H,Y)$ are regular,
then $\sphcl(G \times H, X \cup Y)$ is regular.
\item[(ii)] If $\sphcs(G,X)$ and $\sphcs(H,Y)$ are rational,
then $\sphcs(G \times H, X \cup Y)$ is rational.
\end{description}
\end{proposition}

\begin{proof}
The \wsphcl\ for the direct product group $G \times H$,
viewed as the group generated by $G$ and $H$ (and hence
by $X \cup Y$) with relations $[g,h]=1$ for all $g \in G$
and $h \in H$, is given by
$\sphcl(G \times H, X \cup Y) = \sphcl(G,X) \sphcl(H,Y)$.
Thus (i) follows from the fact that regular
languages are closed under concatenation.
For part (ii), the \wsphcs\ are related by the
formula 
$\sphcs(G \times H, X \cup Y) = \sphcs(G,X) \sphcs(H,Y)$.
\end{proof}

Rivin~\cite[Theorem~14.6]{rivin} has shown that the 
\sphr\ conjugacy growth series of the free group on 
$k$ generators, with respect to a free basis, is not
a rational function.  Combining this with the fact that 
growth series
of regular languages are rational, then the
\sphr\  conjugacy language for any free group,
again with respect to a free basis, cannot be regular. 
In Proposition~\ref{prop:freegpscl} we give a proof that this 
language is not even context-free for the free group 
on two generators, which immediately
extends to the case of a group with a free factor in
Corollary~\ref{cor:freesubscl}.
(See~\cite{hu} for an exposition of the
theory of context-free languages.)

\begin{proposition}\label{prop:freegpscl}
Let $F=F(a,b)$ be a free group on generators $a, b$.
Then the \sphr\  conjugacy language $\sphcl(F, \{a^{\pm 1}, b^{\pm 1}\})$ 
is not context-free.
\end{proposition}

\begin{proof}
Suppose that $\sphcl=\sphcl(F, \{a^{\pm 1}, b^{\pm 1}\})$ 
is context-free and consider the intersection $I=\sphcl \cap L$, 
where $L=a^+b^+a^+b^+$. Since $L$ is a regular language, and the 
intersection of a context-free language with a regular language 
is context-free (see \cite[p.~135, Theorem 6.5]{hu}), 
$I$ is context-free.  Suppose that $a<b$. Then all words 
in $I$ have the form
\begin{equation} \label{form}
a^p b^l a^q b^j \ \ \textrm{with} \ \ p\geq q, 
\end{equation}
and $I$ can be written as the union of the two disjoint sets
\begin{itemize}
\item[] $I_1=\{a^p b^l a^p b^j \mid p,l,j >0, l \leq j\}$
\item[] $I_2=\{a^p b^l a^q b^j \mid  p,l,j >0, p>q>0 \}$
\end{itemize}

Now let $k$ be the constant given by the pumping lemma (stated
below; see~\cite[Lemma~6.1, p.~125]{hu} for details) for 
context-free languages applied to the set $I$, and consider 
the word $W=a^n b^n a^n b^n$, where $n>k$.
One can see $W$ as composed of four 
blocks, the first block being $a^n$, the second being $b^n$ etc. 
Then by the pumping lemma $W$ can be written as $W=uvwxy$, 
where $l(vx)\geq 1$, $l(vwx) \leq k$, and $uv^iwx^iy  \in I$ 
for all $i \geq 0$. Since $l(vwx) \leq k<n$, $vwx$ cannot be 
part of more than two consecutive blocks.

In a first case, suppose that $vwx$ is just part of one block, 
i.e. $vwx$ is a 
power of $a$, or a power of $b$. If $vwx$ is in the first block, 
for $i=0$ one obtains a word that does not satisfy 
(\ref{form}). If it is in the second block then for $i>2$ one gets a 
word $uv^iwx^iy$ of the form $a^p b^l a^p b^j$, but $l >j$, 
so this word doesn't belong to either $I_1$ or $I_2$.
 If $vwx$ is in the third block, for $i>2$ $uv^iwx^iy$ does 
not have the form (\ref{form}).
  If $vwx$ is in the fourth block, for $i=0$ $uwy$ has the 
form $a^p b^l a^p b^j$, $j<l$, and so $uwy$ does not belong to $I$.

In a second case, suppose $vwx$ contains both $a$'s and $b$'s. 
If one of $v$ or $x$ 
contains both $a$ and $b$, then for $i>2$ the word 
$uv^iwx^iy$ contains many blocks alternating between 
powers of $a$ and $b$, and so
does not lie in $a^+b^+a^+b^+$. So $v$ has to be a power 
of one letter only, and $x$ a 
power of the other letter. If $v$ is in first block and 
$x$ in the second block, take 
$i=0$, and one gets a contradiction to (\ref{form}). 
If $v$ is in the second block and $x$ in the third, 
then for $i>2$, $uv^iwx^iy$ has the form $a^p b^l a^q b^j$ 
with $p<q$, which gives a contradiction to (\ref{form}) .
Finally, if $v$ is in the third and $x$ in fourth block, 
for $i>2$ the word $uv^iwx^iy$ does not satisfy (\ref{form}).

Hence none of these cases hold, giving the required contradiction.
\end{proof}

In fact, whenever a group $G$ with a finite inverse-closed
generating set contains a free subgroup $F$ 
such that there is a free basis $\{a_1,...,a_n\}$ of $F$ lying in $X$
and such that all elements of $\sphcl(F,\{a_i^{\pm 1}\})$
are also shortlex conjugacy normal forms for the pair $(G,X)$,
the proof above can be applied.  As a result, we obtain the following.

\begin{corollary}\label{cor:freesubscl}
Let $F$ be a free group with free basis $Z$ and
let $H$ be any group with finite
inverse-closed generating set $Y$.  
\begin{description}
\item[(i)] For the direct product group $F \times H$ the \wsphcl\ 
$\sphcl(F \times H,Z \cup Z^{-1} \cup Y)$ is not a
context-free language.
\item[(ii)] For the free product group $F * H$ the \wsphcl\ 
$\sphcl(F * H,Z \cup Z^{-1} \cup Y)$ is not a
context-free language.
\end{description}
\end{corollary}

%
%
%

Since the \wsphcs\ for the 
free product of two infinite cyclic
groups is not rational, it is natural to
consider next the free product of two finite groups.

\begin{theorem}\label{thm:freeprodfinite}
For finite nontrivial groups $A$ and $B$ with generating sets $X_A=A\setminus 1_A$ 
and $X_B=B\setminus 1_B$, the free product group $A* B$ with generating set 
$X:=X_A \cup X_B$ has 
rational \wsphcs\ $\sphcs(A * B,X)$ 
if and only if $A=B=\mathbb{Z}/{2\mathbb{Z}}$.
Moreover, given any ordering of $X$ satisfying $a<b$ for all $a \in X_A$ and
$b \in X_B$, for the induced shortlex ordering
the associated \wsphcl\ $\sphcl(A * B,X)$ is regular 
if and only if $A=B=\mathbb{Z}/{2\mathbb{Z}}$.
\end{theorem}

\begin{proof}
Let $G:=A * B$ and 
let us assume that all letters in $X_A$ come before the letters in $X_B$ 
in a fixed ordering of $X$.  Notice that the set $\geol(G,X)$ of geodesics in
$G$ are simply the alternating words in $X_A$ and $X_B$.
To simplify notation, write $\sphcl_G:=\sphcl(G,X)$, $\sphcl_A:=\sphcl(A,X_A)$,
and $\sphcl_B:=\sphcl(B,X_B)$.
Then $\sphcl_A \subseteq X_A \cup \{\la\}$ and $\sphcl_B \subseteq X_B \cup \{\la\}$,
where $\la$ denotes the empty word. 
Since any word alternating over $X_A$ and $X_B$ can be cyclically conjugated
to a word of the same length starting with a letter in $X_A$, we have
$$\sphcl_G=\sphcl_A \cup \sphcl_B \cup  \sphcl_{AB},$$
where $\sphcl_{AB}$ is defined to be the set of words in $\sphcl_G$ 
that alternate between $X_A$ and $X_B$,  starting with a letter from $X_A$ 
and ending with a letter from $X_B$. 

Suppose first that $A=B=\Z/2\Z$.
Then $X_A=\{a\}$ and $X_B=\{b\}$ are singleton sets, and
$\sphcl_G=\{\la,a,b\} \cup \{(ab)^r \mid r \in \N\}$.
Hence $\sphcl_G$ is regular, and $\sphcs_G$ is rational.

For the remainder of this proof we assume that at least
one of the groups $A$ or $B$ has order at least 3.  
In order to analyze the \sphr\  conjugacy growth series of $G$, we follow the ideas 
in Theorems 14.2, 14.4 and 14.6 in \cite{rivin}.
We refer the reader to~\cite{conway} for details of 
complex analytic techniques used here.

Notice that all words in 
$\sphcl_{AB}$ have even length; we denote 
those words in  $\sphcl_{AB}$ of length $2r$ by $\sphcl_{AB,2r}$. 
Moreover, if we denote
the subset of $\geol(G,X)$ of alternating words of length $2r$
beginning with a letter in $X_A$ and ending with a letter 
in $X_B$ by $\geol_{AB,2r}$, and similarly for $\geol_{BA,2r}$, then
$|\geol_{AB,2r} \cup \geol_{BA,2r}|=2|X_A|^r|X_B|^r$ and the set
$\sphcl_{AB,2r}$ is in bijective correspondence with the
cyclic conjugacy classes of $\geol_{AB,2r} \cup \geol_{BA,2r}$.
That is, if we let $f(2r)=|\sphcl_{AB,2r}|$,
then $f(2r)$ is the number of orbits of the group 
$\mathbb{Z}/2r\mathbb{Z}$ acting by cyclic conjugation on the set
$\geol_{AB,2r} \cup \geol_{BA,2r}$. We can compute
$f(2r)$ by using Burnside's Lemma, which gives
\begin{eqnarray*}
f(2r)&=&\frac{1}{2r}\sum_{g \in \mathbb{Z}/2r\mathbb{Z}} |Fix(g)|\\
&=&\frac{1}{2r}\sum_{2|d|2r} ~~\sum_{g \in \Z/2r\Z, \gcd(g,2r)=d} |Fix(g)|\\
\end{eqnarray*}
where the second equality uses the fact that an
odd element of $\Z/2r\Z$ cannot permute an even length alternating
word to itself.
Now for any $g \in \Z/2r\Z$ with $2|d=gcd(g,2r)$
and any alternating word $w \in Fix(g)$, we have $w=v^{2r/d}$ for
an alternating word $v \in \geol_{AB,d} \cup \geol_{BA,d}$.
There are $2|X_A|^{d/2}|X_B|^{d/2}$ such words.  Also using the
fact that $\phi(\frac{2r}{d})=|\{g \le 2r \mid \gcd(g,2r)=d\}|$,
where $\phi$ is the Euler totient function, yields 
\begin{eqnarray*}
f(2r)
&=&\frac{1}{r}\sum_{2|d|2r} \phi(\frac{2r}{d})|X_A|^{d/2}|X_B|^{d/2}\\
&=&\frac{1}{r}\sum_{e|r} \phi(\frac{r}{e})|X_A|^{e}|X_B|^{e}\\
&=&\frac{1}{r}\sum_{e|r} \phi(e)|X_A|^{r/e}|X_B|^{r/e}~.\\
\end{eqnarray*}

Let $\sphcs_G:=\sphcs(G,X)$.  Now we have
\begin{eqnarray*}
\sphcs_G(z)&=& 1 + |\sphcl_A \cup \sphcl_B \setminus \{\la\}|z+
              \sum_{r=1}^\infty |{\sphcl_{AB,2r}}|z^{2r}\\
&=& 1 + |\sphcl_A \cup \sphcl_B \setminus \{\la\}|z+
  \sum_{r=1}^\infty \frac{1}{r}\sum_{e|r} \phi(e)(|X_A||X_B|)^{r/e} z^{2r}~.\\
\end{eqnarray*}
To simplify notation, let $\alpha:=|X_A||X_B|$ and 
$\beta:=|\sphcl_A \cup \sphcl_B \setminus \{\la\}|$.
Formally taking the derivative of this power series and multiplying by $z$ gives
\[z\sphcs_G'(z)
=\beta z+
  2\sum_{r=1}^\infty \sum_{e|r} \phi(e)\alpha^{r/e} (z^{2})^r~.\]
Rearrange the terms of this formal power series
(e.g. as in \cite[Theorem~14.4]{rivin} with $c_n=\phi(n)$ and $b_n=\alpha^{n}$)
to obtain
\[z\sphcs_G'(z)=\beta z+
  2\sum_{d=1}^\infty  \phi(d) \sum_{n=1}^\infty \alpha^{n} (z^{2d})^n~.\]

Now $\alpha=|X_A||X_B|$ and at least one 
of the sets $X_A,X_B$ contains more than one element.
Hence $\alpha>1$.  For any real number $r>0$ and any point $p \in \C$,
let $D_r(p)$ denote
the open disk of the  complex plane
defined by $D_r(p):=\{z \mid |p-z|<r\}$.  Since 
$\sum_{n=1}^\infty \alpha^{n} (z^{2d})^n=\frac{\alpha z^{2d}}{1-\alpha z^{2d}}$
for all $z$ in the disk $D_{1/\alpha}(0)$,
we have $z\sphcs_G'(z)=h(z)$ in this disk,
where
\[h(z):=\beta z+
  2\sum_{d=1}^\infty  \phi(d) \frac{\alpha z^{2d}}{1-\alpha z^{2d}}~.\]
  
For any $0<\epsilon<1$ there are only finitely many
even roots of $\frac{1}{\alpha}$ in the closed disk 
$\overline{D_{1-\epsilon}(0)}$; denote these roots by
$z_1,...,z_k$.
(In particular, all $2d$-th roots of $\frac{1}{\alpha}$
in this closed disk
must satisfy $d \le -\ln(\alpha)/2\ln(1-\epsilon)$.)
Let 
\[\epsilon':=\frac{\epsilon}{3} \min\left[ \{1-\epsilon-|z_i| \mid 1 \le i \le k\}
  \cup   \{d(z_i,z_j) \mid i \neq j\} \right],
\]
and let $R$ be the region of the complex plane defined by
$R:=\overline{D_{1-\epsilon}(0)} \setminus \cup_{i=1}^k D_{\epsilon'}(z_i)$.
Then for all $z \in R$ the value of $|1-\alpha z^{2d}|$ must
be strictly greater than $0$.  Since $R$ is compact, then
there is a $\delta>0$  such that for all $z$ in
the region $R$, 
we have
$|1-\alpha z^{2d}| \ge \delta$.
Now 
$|\phi(d) \frac{\alpha z^{2d}}{1-\alpha z^{2d}}| \le  
d \frac{\alpha}{\delta} (1-\epsilon)^{2d}$.
Since the series
$\sum_{d=1}^\infty d \frac{\alpha}{\delta} (1-\epsilon)^{2d}$
converges to a finite number, the Weierstrass M-test~\cite[II.6.2]{conway}
says that the series $h(z)$ is uniformly convergent on this region $R$.
Since each partial sum in the expression
for $h$ is continuous on $R$, then $h$ is also continuous
on this region~\cite[II.6.1]{conway}, and since each partial sum
is analytic, the function $h$ is also analytic on $R$~\cite[VII.2.1]{conway}.
Allowing $\epsilon$ to shrink to 0, this shows that
$h$ is analytic on the disk $D_1(0)$ outside of the infinite set of points
that are $2d$-th roots of $\frac{1}{\alpha}$.
Moreover, whenever $y$ is a $2d_0$-th root of $\frac{1}{\alpha}$, 
a similar argument shows that the sum
$\sum_{d=1,d \neq d_0}^\infty  \phi(d) \frac{\alpha y^{2d}}{1-\alpha y^{2d}}$
is analytic, and so the function $h$ has a pole at the point $y$.
That is, $h(z)$ is analytic on the unit disk except for
infinitely many poles.

Now assume that the power series $z \sphcs_G'(z)$ is a rational
function $p(z)$.
Then $p$ must be analytic in the unit disk
outside of finitely many poles.  Let $R'$ be the unit disk
with all of the poles of both $h$ and $p$ removed.
Then we have $p$ and $h$ are analytic on $R'$, and
$p=h$ in an open disk $D_{1/\alpha}(0)$; thus 
$p=h$ on $R'$~\cite[IV.3.8]{conway}.
This shows that infinitely many of the poles of $h$
must be removable singularities, which is a contradiction.

Thus $z\sphcs_G'(z)$ cannot be
a rational function.  Therefore the function 
$\sphcs_G'$ also is not rational.  Since the
derivative of any rational function is also rational,
this shows that $\sphcs_G$ 
also is not a rational function.  Since the growth series
of the set $\sphcl_G$ is not rational in this case, we must also
have that this set is not a regular language.
\end{proof}

\section{Geodesic conjugacy series and languages}\label{sec:geodesic}

Given a finite simplicial graph with a group attached to each vertex, 
the associated \textit{graph product} is the group generated by  
the vertex groups with the added relations that elements of distinct 
adjacent vertex groups commute.
This construction often preserves geometric, algebraic or algorithmic 
properties of groups.  In particular Loeffler, Meier, and 
Worthington~\cite{lmw} have shown that regularity of the full
language $\geol$ of geodesics (with respect to the union of the generating
sets for the vertex groups) is preserved by the graph product
construction.  
Theorem~\ref{thm:graphprodgcl} 
gives a further illustration of this behavior,
showing that regularity of the pair of
languages $\geol$ and $\geocl$ 
is preserved by graph products. 

\begin{theorem}\label{thm:graphprodgcl}
If $G$ is a graph product of groups $G_i$ with $1 \le i \le n$, and
each $G_i$ has a finite inverse-closed generating set $X_i$ such
that both $\geol(G_i,X_i)$ and $\geocl(G_i,X_i)$ are regular,
then both the \wgeol\ $\geol(G,\cup_{i=1}^n X_i)$ and
the \wgeocl\ 
$\geocl(G,\cup_{i=1}^n X_i)$ are also regular.
\end{theorem}

Before proceeding with the proof of Theorem~\ref{thm:graphprodgcl}, we need  
some preliminary notation and results.
Let $\lm$ be a finite simplicial graph with $n$ vertices $v_1,...,v_n$
and suppose that for each $1 \le i \le n$ the vertex $v_i$ is labeled by
a group $G_i$ that has a finite inverse-closed generating set $X_i$ 
with \wgeol\ $\geol_i:=\geol(G_i,X_i)$ and \wgeocl\ $\geocl_i:=\geocl(G_i,X_i)$.
(Note that two vertices of $\lm$ are considered to be adjacent here if
the vertices are distinct and joined by an edge.)  
For two words $u,w \in X_i^*$, we write
$u \ei w$ if $u$ and $w$ represent the same element of $G_i$,
and $u \sim_i w$ if $u$ and $w$ represent conjugate elements of $G_i$. 
Let $G$ be the associated graph product
with generating set $X:=\cup_{i=1}^n X_i$.  For
words $y,z \in X^*$, write $y=_Gz$ if $y$ and $z$ represent
the same element of $G$, and $y \sim_G z$ if $y$ and $z$ represent
conjugate elements of $G$.
Also let $\geol$ and $\geocl$ denote the 
\wgeol\ and \wgeocl, respectively, for $G$ over $X$.

For each $i$, we define the {\em centralizing set} $C_i$ to be
the union of the sets $X_j$ such that the vertices $v_i$ and
$v_j$ are adjacent in the graph $\lm$. 
Given a word $w=a_1 \cdots a_m$ with each $a_i \in X$, 
the {\em centralizing set} $C(w)$ associated to $w$
is defined by $C(w):=\cap_{i=1}^m C_{j_i}$, where
for each $1 \le i \le m$, the letter $a_i$ lies in
the set $X_{j_i}$.  That is, $C(w)$ is the subset
of $X$ that commutes with every letter of $w$ from
the graph product construction.

We define several types of rewriting operations on
words over $X$ as follows.
\begin{description}
\item[(x0)] {\em Local reduction}: $yuz \ra ywz$ with $y,z \in X^*$,
$u,w \in X_i^*$, $u=_{G_i}w$, and $l(u)>l(w)$.
\item[(x1)] {\em Local exchange}: $yuz \ra ywz$ with $y,z \in X^*$,
$u,w \in X_i^*$, $u=_{G_i}w$, and $l(u)=l(w)$.
\item[(x2)] {\em Shuffle}: $yuwz \ra ywuz$ with $y,z \in X^*$,
$u \in X_i^*$, $w \in X_j^*$, and $v_i,v_j$ adjacent in $\lm$.
\item[(x3)] {\em Cyclic conjugation}: $yz \ra zy$ with $y,z \in X^*$.
\item[(x4)] {\em Conjugacy exchange}: $uy \ra wy$ with $y \in X^*$,
$u,w \in X_i^*$, $X_i \subseteq  C(y)$, $u \sim_i w$, and $l(u)=l(w)$.
\item[(x5)] {\em Conjugacy reduction}: $uy \ra wy$ with $y \in X^*$,
$u,w \in X_i^*$, $X_i \subseteq  C(y)$, $u \sim_i w$, and $l(u)>l(w)$.
\end{description}

For $0 \le j \le 5$, we write $y \sr{xj}{\ra} z$ if
$y$ is rewritten to $z$ with a single application of operation (x$j$).
Given a subset $\alpha$ of $0-5$, we write $y \sr{x\alpha *}{\ra} z$
if $z$ can obtained from $y$ by a finite (possibly zero) number
of rewritings of the types (x$j$) with $j$ in the set $\alpha$.  
Note that rewriting operations (x1)-(x4) preserve word
length, and (x0), (x5) decrease word length.

A word $y \in X^*$ is {\em \tri} if whenever $y \sr{x0-2*}{\ra} z$,
no operation of type (x0) can ever occur.  The word $y$ is
{\em \ct} if whenever $y \sr{x0-5*}{\ra} z$, no operation 
of type (x0) or (x5) can occur.

For each $i$, we define a monoid homomorphism
$\pi_i:X^* \rightarrow (X_i \cup \{\$\})^*$, 
where $\$$ denotes a letter not in $X$, by
defining 
\[\pi_i(a) := \begin{cases}
a & \mbox{if } a \in X_i \\
\$& \mbox{if } a \in X \setminus(X_i \cup C_i) \\
1 & \mbox{if } a \in C_i\\
\end{cases}\]

Given any subset $t$ of $X$, we define the {\it support}
$supp(t)$ of $t$ to be the set of all vertices $v_i$ of $\lm$ such that
$t$ contains an element of $X_i$. Let 
$T$ be the set of all nonempty subsets $t$ of $X$ satisfying the
properties that $supp(t)$ is a clique of $\lm$, and for each $v_i \in supp(t)$, 
the intersection $t \cap X_i$ is a single element of $X_i$.
Note that for each $t=\{a_1,...,a_k\} \in T$, whenever
$a_{i_1},...,a_{i_k}$ is another arrangement of the
letters in $t$, then
$a_1 \cdots a_k =_G a_{i_1} \cdots a_{i_k}$.
Hence $t$ denotes a well-defined element of $G$.
Also note that for each $a \in X$, we have $\{a\} \in T$,
and $a$ is the element of $G$ associated to $\{a\}$.
By slight abuse of notation, we will consider
$X \subseteq T \subseteq G$.  Now $T$ is another
inverse-closed generating set for $G$.

\begin{example}\label{ex:t}
{\em
For the graph product of three 
infinite cyclic groups $G_i=\langle a_i \rangle$ ($1 \le i \le 3$),
if the only adjacent pair of vertices is $v_2,v_3$, 
then the graph product group $G=G\lm=G_1 * (G_2 \times G_3)$
has generating sets \\
$X=\{a_1, a_1^{- 1}, a_2, a_2^{- 1}, a_3, a_3^{- 1},\}$
and \\
$T=\{~\{a_1\}, \{a_1^{- 1}\}, \{a_2\}, \{a_2^{- 1}\}, \{a_3\}, \{a_3^{- 1}\},
\{a_2a_3\}, \{a_2^{-1}a_3\}, \{a_2a_3^{-1}\}, \{a_2^{-1}a_3^{-1}\}~\}$.
}
\end{example}

Analogous to the operations above on words over $X$,
we define three sets of rewriting rules on words over $T$
as follows.  
For each index $i$ fix a total ordering on 
$X_i$, and let $<_i$
denote the corresponding shortlex ordering on $X_i^*$.
To ease notation, we
let the empty set $\{ \}$ denote the empty word over $T$.
Whenever $t \in T \cup \{\emptyset\}$, $a \in X$, $t \cup \{a\} \in T$,
and $t \cap \{a\} = \emptyset$, 
we let $\{a,t\}$ denote the set $t \cup \{a\}$.  
\begin{description}

\item[(R0)] 
$\{a_1,t_1\} \cdots \{a_m,t_m\} \ra
\{b_1,t_1\} \cdots \{b_{k},t_{k}\}t_{k+1} \cdots t_m$ whenever
there is an index $1 \le i \le n$ such that for each $1 \le j \le m$,
$a_j,b_j \in X_i$, $t_j \in T \cup \{\emptyset\}$, and
$\{a_j,t_j\} \in T$; $k<m$; and
$a_1 \cdots a_m \ei b_1 \cdots b_{k}$.

\item[(R1)] $\{a_1,t_1\} \cdots \{a_m,t_m\} \ra
\{b_1,t_1\} \cdots \{b_{m},t_{m}\}$ whenever
there is an index $1 \le i \le n$ such that for each $1 \le j \le m$,
$a_j,b_j \in X_i$, $t_j \in T \cup \{\emptyset\}$,
$\{a_j,t_j\} \in T$; 
$a_1 \cdots a_m \ei b_1 \cdots b_m$; and
$a_1 \cdots a_m >_i b_1 \cdots b_m$.

\item[(R2)] $t\{a,t'\} \ra \{a,t\}t'$ whenever
$a \in X$, $t \in T$, $t' \in T \cup \{\emptyset\}$,
and $\{a,t\}, \{a,t'\} \in T$.
\end{description}
Also in analogy with the operations above on $X^*$,
whenever $0 \le j \le 2$, $\alpha \subseteq \{0,1,2\}$ and $w,x \in T^*$,
we write $w \sr{Rj}{\ra} x$ if $w$ rewrites to $x$ via
exactly one application of rule (R$j$), and
$w \sr{R\alpha*}{\ra} x$ if $x$ can be obtained
from $w$ by a finite (possibly zero) number of rewritings
using rules of the type (R$j$) for $j \in \alpha$.

Let $R$ denote the set of all rewriting rules of the
form (R0), (R1), and (R2).  Note that the generating
set $T$ of $G$ together with the relations given by
the rules of $R$ form a monoid presentation of $G$;
hence, $(T,R)$ is a {\it rewriting system} for the
graph product group $G$.  We refer the reader 
to Sims' text \cite{sims} for definitions and
details on rewriting systems for groups which we use 
in this section.

Define a partial ordering on $T$ by $\{a,t\} > \{b,t\}$ whenever
$a,b \in X_i$ for some $i$, $t \in T \cup \emptyset$,
$\{a,t\} \in T$, and $a >_i b$; and by
$\{a,t\} < t$ whenever $a \in X$ and $t, \{a,t\} \in T$.
For each $t$ in $T$, define the {\em weight} $wt(t)$
of $t$ to be the number of elements of $t$ as a
subset of $X$ (equivalently, $wt(t)$ is the
number of vertices in $supp(t)$).
Then all of the
rules in the rewriting system $R$ decrease the associated
weightlex ordering on $T^*$.  Since the weightlex ordering
is compatible with concatenation and well-founded,
no word $w \in T^*$
can be rewritten infinitely many times; that is, the
system $R$ is {\it terminating}.
One can check via the Knuth-Bendix algorithm~\cite[Chapter~2]{sims}
that this system is also {\it confluent}, i.e., that
whenever
a word $w$ rewrites two words $w \ra w'$ and $w \ra w''$
using these rules, then there is a word $w'''$ such
that $w' \sr{R0-2*}{\ra} w'''$ and $w' \sr{R0-2*}{\ra} w'''$;
this proof is provided in Appendix~\ref{app}.
(An alternative proof of this can be found in~\cite[Theorem~C]{hm}.
See Example~\ref{ex:racg} below for details of 
this rewriting system for an example of a right-angled
Coxeter group.)

Hence for each word $w \in X^*$ there is a unique 
word $irr(w)$ in $T^*$ that is irreducible (i.e.~cannot be rewritten)
such that $w \sr{R0-2*}{\ra} irr(w)$, and each
element $g \in G$ is represented by a unique
word $irr(g)$ in $T^*$ that is irreducible 
with respect to the rewriting rules in $R$~\cite[Prop.~2.4, p.~54]{sims}.
That is, the set 
\[ irr(R) := \{irr(g) \mid g \in G\}
\]
is a set of weightlex normal forms for $G$ over $T$.

For each $t \in T$, let $a_1,...,a_k$ be 
a choice of an ordered listing of the elements of
the set $t$, and let $h(t):=a_1 \cdots a_k$.
Then $h$ determines a monoid homomorphism $h:T^* \ra X^*$.
For each $w \in X^*$, let $\tht(w):=h(irr(w))$. 
Then the set 
\[
h(irr(R)) = \{\tht(w) \mid w \in X^*\}
\]
is
a set of normal forms for $G$ over the original
alphabet $X$.

In Proposition~\ref{prop:geoword} below, we  show that 
the geodesic words for the graph product group $G$ over $X$
are exactly the trimmed words, and can be characterized
as an intersection of homomorphic inverse images via these
$\pi_i$ maps.  Although the equivalence of (i), (ii), and (iii)
of this Proposition follows from results in~\cite{hm} and~\cite{lmw}, 
we include some details here for a condensed exposition
and because the further equivalence with (iv) 
will be needed for our proof of Proposition~\ref{prop:geocon} below.

\begin{proposition}\label{prop:geoword}  
Let $G$ be a graph product of the groups $G_i$ for $1 \le i \le n$,
Let $X=\cup_{i=1}^n X_i$ where $X_i$ is a finite inverse-closed
generating set for $G_i$ and $\geol_i:=\geol(G_i,X_i)$
for each $i$,
and let $y \in X^*$.  The following are equivalent.
\begin{description}
\item[(i)] $y$ is a geodesic word for $G$ with respect to $X$.
\item[(ii)] $y$ is a \tri\ word over $X$.
\item[(iii)] $y \in \cap_{i=1}^n \pi_i^{-1}(\geol_i(\$ \geol_i)^*)$.
\item[(iv)] $y \sr{x1-2*}{\ra} \tht(y)$.
\end{description}
\end{proposition}

\begin{proof}
(i) $\Rightarrow$ (ii): If $y$ is not \tri, then $y$
can be rewritten using length-preserving operations
$y \sr{x1-2*}{\ra} z$ to a word $z$ that can be further rewritten
using a length-reducing operation of type (x0).  Hence
$y$ cannot be a geodesic.

\smallskip

\noindent (ii) $\Rightarrow$ (iii):  Suppose that
$\pi_i(y) \in (X_i \cup \{\$\})^* \setminus (\geol_i(\$ \geol_i)^*)$
for some $i$.
Then $\pi_i(y)=y_0\$ y_1 \cdots \$y_k$ for some $k \ge 0$ 
and each $y_j \in X_i^*$, where for some $j$ we have $y_j \notin \geol_i$.
Then $y \sr{x2*}{\ra} y'y_jy''$, and a local reduction (operation of type (x0))
may be applied to the latter.  Hence $y$ is not \tri.

\smallskip

\noindent (iii) $\Rightarrow$ (iv): 
Suppose that $y \in  \cap_{i=1}^n \pi_i^{-1}(\geol_i(\$ \geol_i)^*)$.

Each rule (R$j$) of the rewriting system $R$ gives rise to
a commutative diagram via the map $h$ with a sequence of
operations on words over $X$.
In particular, for any words $w,x \in T^*$ with $w \sr{R0-2}{\ra}x$, we have
\[ 
\begin{array}{ccccccccccc}
w & \sr{R0}{\lra} & x & \hspace{.1in} & w & \sr{R1}{\lra} & x 
& \hspace{.1in} & w & \sr{R2}{\lra} & x \\
h \downarrow & & \downarrow h & & h \downarrow & & \downarrow h 
& & h \downarrow & & \downarrow h \\
h(w) & \sr{x2*}{\ra} \sr{x0}{\ra} \sr{x2*}{\ra} & h(x) & &
h(w) & \sr{x2*}{\ra} \sr{x1}{\ra} \sr{x2*}{\ra} & h(x)
& & h(w) & \sr{x2*}{\lra} & h(x)
\end{array} 
\]

For $1 \le i \le n$ and $0 \le j \le 1$
we can refine the operation (x$j$) by
defining the rewriting
operation (x$j$$i$)
to denote an operation of type (x$j$) in which a subword 
over $X_i$ is rewritten, and similarly we 
let rewriting rule
(R$j$$i$) denote a rule of type (R$j$) in which
effectively an $X_i^*$ subword is rewritten.

Using these refined operations and the maps $\pi_i$, 
we can extend the above diagrams 
to words over $X_i \cup   \{\$\}$ as follows.  We say that
a rewriting operation $yuz \ra yvz$, for $y,z \in (X_i \cup \{\$\})^*$
and $u,v \in X_i$ with $u \ei v$, is of 
{\it type} {\bf (s0$i$)} if $l(u)>l(v)$ and of 
{\it type} {\bf (s1$i$)} if $l(u)=l(v)$.  
Then for any words $w,x \in X^*$ with $w \sr{x0-2}{\ra} x$, 
and any $j \in \{0,1\}$ 
and $1 \le i,k \le n$ with $i \neq k$, we have
\[ 
\begin{array}{ccccccccccc}
w & \sr{xji}{\lra} & x & 
\hspace{.1in} & w & \sr{x1k}{\lra} & x 
& \hspace{.1in} & w & \sr{x2}{\lra} & x \\
\pi_i \downarrow & & \downarrow \pi_i & & \pi_i \downarrow & & \downarrow \pi_i 
& & \pi_i \downarrow & & \downarrow \pi_i \\
\pi_i(w) & \sr{sji}{\lra} & \pi_i(x) & &
\pi_i(w) & \sr{id}{\lra} & \pi_i(x)
& & \pi_i(w) & \sr{id}{\lra} & \pi_i(x)
\end{array} 
\]
where $id$ denotes the identity map on $(X_i \cup \{\$\})^*$.  Hence for
any words $w,x \in T^*$ with $w \sr{R0-2}{\ra} x$ and any $j \in \{0,1\}$ 
and $1 \le i,k \le n$ with $i \neq k$, we have
\[ 
\begin{array}{ccccccccccc}
w & \sr{Rji}{\lra} & x & \hspace{.1in} & w & \sr{R1k}{\lra} & x 
& \hspace{.1in} & w & \sr{R2}{\lra} & x \\
\pi_i \circ h \downarrow & & \downarrow \pi_i \circ h 
& & \pi_i \circ h \downarrow & & \downarrow \pi_i \circ h 
& & \pi_i \circ h \downarrow & & \downarrow \pi_i \circ h \\
\pi_i(h(w)) & \sr{sji}{\lra} & \pi_i(h(x)) & &
\pi_i(h(w)) & \sr{id}{\lra} & \pi_i(h(x))
& & \pi_i(h(w)) & \sr{id}{\lra} & \pi_i(h(x))~.
\end{array} 
\]

For our word $y \in \cap_{i=1}^n \pi_i^{-1}(\geol_i(\$ \geol_i)^*)$, 
the inclusion map $\iota:X^* \ra T^*$ 
allows us to consider $y=\iota(y)$ as a word over the
alphabet $T$, where $h(\iota(y))=h(y)=y$.
Since $R$ is a terminating and confluent rewriting system, we
have $y \sr{R0-2*}{\ra} irr(y)$, and so by
the commutative diagrams above, $y \sr{x0-2*}{\ra} \tht(y)$.

Suppose that an operation of type (x0) appears 
in this sequence of rewritings.  Then
$y \sr{x1-2*}{\ra} y' \sr{x0i}{\ra} z \sr{x0-2*}{\ra} \tht(y)$
for some $y',z \in X^*$ and some $1 \le i \le n$.
Again applying the commutative diagrams above,
then 
$\pi_i(y) \sr{s1i*}{\ra} \pi_i(y') \sr{s0i}{\ra} \pi_i(z)$.
However, operations of type (s1$i$) map elements of
$\geol_i(\$ \geol_i)^*$ to $\geol_i(\$ \geol_i)^*$,
so  $\pi_i(y') = \in 
\geol_i(\$ \geol_i)^*$ and no operation of type
(s0$i$) can be applied to $\pi_i(y')$, giving a contradiction.

\smallskip

\noindent (iv) $\Rightarrow$ (i):  Suppose that
$y \sr{x1-2*}{\ra} \tht(y)$.  Then $l_X(y) = l_X(\tht(y))$.
If $z$ is the shortlex least representative over $X$ of the
same element of $G$ as $y$, then since the set
$\{\tht(w)\}$ is a set of normal forms, we have
$\tht(z)=\tht(y)$.  Now $z=\iota(z) \sr{R0-2*}{\ra} irr(z)$,
and so by the argument above we have $z=h(z) \sr{x0-2*}{\ra} h(irr(z))=\tht(z)$.
However, since $z$ is geodesic, no length-decreasing 
operations can apply, so we have
$l_X(z)=l_X(\tht(z))=l_X(y)$.  Therefore $y$ is also geodesic.
\end{proof}

In the next Corollary we collect for later use 
two other results that
follow from the 
proof of Proposition~\ref{prop:geoword}.

\begin{corollary}\label{cor:geoword}
Using the notation above:
\begin{enumerate}
\item The subset
$h(irr(R))=\{\tht(w) \mid w \in X^*\} \subseteq X^*$ 
is a set of geodesic normal forms for $G$ over $X$.
\item For any word $w \in X^*$ there is a sequence
of rewriting operations $w \sr{x0-2*}{\ra} \tht(w)$.
\end{enumerate}
\end{corollary} 

\begin{proof}
Statement (1) is shown in the proof of (iv) $\Rightarrow$ (i) above.
For (2), let $w$ be any element of $X^*$.
Using the inclusion map $\iota:X^* \ra T^*$, we have
$w=\iota(w) \sr{R0-2*}{\ra} irr(w)$, and 
so from the first set of
commutative diagrams in the proof of (iii) $\Rightarrow$ (iv) in 
Proposition~\ref{prop:geoword} above,
we have $w=h(\iota(w)) \sr{x0-2*}{\ra} \tht(w)$.
\end{proof}

In the following Proposition
we show that a result similar to Proposition~\ref{prop:geoword}
holds for {\geocon}s.

\begin{proposition}\label{prop:geocon}  
Let $G$ be a graph product of the groups $G_i$ for $1 \le i \le n$ and
let $X=\cup_{i=1}^n X_i$ where $X_i$ is a finite inverse-closed
generating set for $G_i$ for each $i$.  
Also for each $i$ let $\geol_i:=\geol(G_i,X_i)$, 
$\geocl_i:=\geocl(G_i,X_i)$, and 
\[ \widetilde U_i :=  \{u_0 \$ u_1 \cdots \$ u_m \mid 
  m \ge 1 \mbox{ and } u_1,...,u_{m-1},u_mu_0 \in \geol_i\},
\]
and let $y \in X^*$.  The following are equivalent.
\begin{description}
\item[(i)] $y$ is a \geocon\  for $G$ with respect to $X$.
\item[(ii)] $y$ is a \ct\ word over $X$.
\item[(iii)] $y \in \cap_{i=1}^n \pi_i^{-1}(\geocl_i \cup \widetilde U_i)$.
\end{description}
\end{proposition}

\begin{proof}
(i) $\Rightarrow$ (ii):  Suppose that $y \in X^*$ is not \ct.  Then
$y \sr{x1-4*}{\ra} z$ for some word $z$ that can be rewritten
using a length-reducing operation of type (x0) or (x5) to a word
representing an element of the same conjugacy class.  Since $l(y)=l(z)$, then
$y$ cannot be a \geocon.

\smallskip

\noindent (ii) $\Rightarrow$ (iii):  Suppose that
$y \notin \cap_{i=1}^n \pi_i^{-1}(\geocl_i \cup \widetilde U_i)$. 
Then there is an index $i$ such that 
$\pi_i(y) \in (X_i \cup \{\$\})^* \setminus (\geocl_i \cup \widetilde U_i)$.
If $\pi_i(y) \in X_i^*$, then all letters of $y$ lie in 
$X_i \cup C_i$,  and so
$y \sr{x2*}{\ra} \pi_i(y)y'$ for some $y' \in C_i^*$.
Now $\pi_i(y) \notin \geocl_i$ implies that a conjugacy
reduction (x5) can be applied to the word $\pi_i(y)y'$, and so
$y$ cannot be \ct\ in this case.  On the other hand,
if $\pi_i(y) \notin X_i^*$, then
$\pi_i(y)=u_0\$ \cdots \$u_m$ such that $m \ge 1$,
each $u_i \in X_i^*$, and at least one of
$u_1,...,u_{m-1},$ or $u_mu_0$ does not lie in $\geol_i$.
A similar argument shows that in this case
$y \sr{x2-3*}{\ra} z$ for a word $z$ to which a 
local reduction (x0) can be applied, implying again that
$y$ cannot be \ct.

\smallskip

\noindent (iii) $\Rightarrow$ (i):
Suppose that $y \in \cap_{i=1}^n \pi_i^{-1}(\geocl_i \cup \widetilde U_i)$,
but that $y$ is not a \geocon.  
(Note that since $\geocl_i \subseteq \geol_i$ and 
$\widetilde U_i \subseteq \geol_i(\$ \geol_i)^*$,
Proposition~\ref{prop:geoword} implies that the word $y$ is a geodesic
for $G$ over $X$.)
Then there is a geodesic word $w \in X^*$
such that the element $wyw^{-1}$ of $G$ is represented by
a word 
over $X$ that is shorter than $y$.
In particular, the result in Corollary~\ref{cor:geoword}(1) 
that the $\tht$ normal forms are geodesics
shows that
the word $\tht(wyw^{-1})$ must then be shorter than $y$.
Choose such a pair of words $y\in \cap_{i=1}^n \pi_i^{-1}(\geocl_i \cup \widetilde U_i)$
and $w \in \geol=\geol(G,X)$ with  $l_X(\tht(wyw^{-1}))<l_X(y)$
such that 
$l_X(y)+2l_X(w)$ is least possible 
among all pairs with these properties.

Using Corollary~\ref{cor:geoword}(2)
we have $wyw^{-1} \sr{x0-2*}{\ra} \tht(wyw^{-1})$
(where $w^{-1}$ is the symbolic inverse of $w$ over $X$).
Since $l_X(\tht(wyw^{-1})) < l_X(y)$,
at least one operation of type (x0) must apply
in this sequence of rewriting operations.

Since property (iii) of Proposition~\ref{prop:geocon} holds for $y$,
for each $1 \le i \le n$
we can write
$\pi_i(y)=y_{i}u_iy_{i}'$ with
either $y_{i} \in \geocl_i$ and $u_i=y_i'=\la$ (where
$\la$ denotes the empty word), or
$u_i \in \$(\geol_i \$)^*$
and $y_i'y_i \in \geol_i$.  
Since $w$ is a geodesic, then
for each $1 \le i \le n$, using Proposition~\ref{prop:geoword}
we can write 
$\pi_i(w) = v_iw_i$ 
with $v_i \in (\geol_i \$)^*$ and
$w_i \in \geol_i$.  
Then 
$\pi_i(wyw^{-1}) = v_i w_{i} y_{i}u_iy_i' w_{i}^{-1} v_i^{-1}$
(where the formal inverse of a word 
$b_{1}\$ \cdots b_{k}\$$ with each $b_{j} \in X_i^*$
is defined to be $\$b_{k}^{-1} \cdots \$b_{1}^{-1}$).

In our commutative diagrams in the proof above, we did not 
consider the interaction of rewriting operations of type (x0$i$) with the 
map $\pi_k$ when $k \neq i$, but we need to do so now.
For any word 
$s \in (X_j \cup \$)^*$, we
write $s \sr{t*}{\ra} s'$ if $s'$ can be obtained
from $s$ by finitely many (possibly zero) replacements $\$\$ \ra \$$; i.e.
by shortening but not eliminating subwords that are strings of 
$\$$ signs.  Suppose that $z \in X^*$ and 
$\pi_i(z)=abc$ with $a \in (X_i^*\$)^*$, $b \in X_i^*$,
and $c \in (\$X_i^*)^*$.  Suppose that the
operation $b \sr{x0i}{\ra} b'$ induces the operation
$z \sr{x0i}{\ra} z':=ab'c$, and that $\la \neq b' \in X_i^*$.  If
the vertex $v_j$ is adjacent to $v_i$ in the graph $\lm$, and the vertex $v_k$
is not adjacent to $v_i$, then
\[ 
\begin{array}{ccccccc}
z & \sr{x0i}{\lra} & z' 
& \hspace{.3in} & z & \sr{x0i}{\lra} & z' \\
\pi_j \downarrow & & \downarrow \pi_j & & \pi_k \downarrow & & \downarrow \pi_k \\ 
\pi_j(z) & \sr{id}{\lra} & \pi_j(z') & &
\pi_k(z) & \sr{t*}{\lra} & \pi_k(z')
\end{array} 
\]

\medskip

In the following claim, we use these
commutative diagrams to show that in the process
of rewriting $wyw^{-1} \sr{x0-2*}{\ra} \tht(wyw^{-1})$,
each time an (x0$i$) rewriting operation is applied,
then on the level of images under the $\pi_j$ maps,
when $j \neq i$ the effect is the application of a
$\sr{t*}{\ra}$ operation, and when $j=i$, effectively either
a rewrite of the $w_i y_i w_i^{-1}$ subword of $wyw^{-1}$
is replaced by a word of length at least
$l(y_i)$ (in the case that $\pi_i(y) \in \geocl_i$)
or rewrites of the $w_i y_i$ and $y_i' w_i^{-1}$ 
subwords of $wyw^{-1}$
are replaced by words of length at least
$l(y_i)+l(y_i')$ (otherwise).

\noindent {\it Claim:} Suppose that
$wyw^{-1} \sr{x0-2*}{\ra} z \sr{x0i}{\ra} x
\sr{x0-2*}{\ra} \tht(wyw^{-1})$ and
for each $1 \le j \le n$ the word $z \in X^*$ satisfies either:
\begin{description}

\item[(a)] In the case that $\pi_j(y) =y_j \in \geocl_j$: \\ 
           $\pi_j(z)=\tilde v_j z_j \tilde v_j'$  for some\\           
           $\tilde v_j \in (\geol_j\$)^*$,
           $\tilde v_j' \in (\$ \geol_j)^*$, 
           and  $z_j \in X_j^*$ such that \\
           $v_j \sr{t*}{\ra} \tilde v_j$, $v_j^{-1} \sr{t*}{\ra} \tilde v_j'$, 
           $z_j =_{G_j} w_jy_{j}w_j^{-1}$ and    \\
           $l(y_j) \le l(z_j) \le l(y_j)+2l(w_j)$, \\
\hspace{.1in} or
\item[(b)] In the case that $\pi_j(y) \in \widetilde U_j$: \\ 
           $\pi_j(z)=\tilde v_jz_j\tilde u_jz_j'\tilde v_j'$ for some \\
           $\tilde u_j \in \$(\geol_j\$)^*$,
           $\tilde v_j \in (\geol_j\$)^*$,
           $\tilde v_j' \in (\$ \geol_j)^*$, 
           and $z_j,z_j' \in X_j^*$ such that \\
           $u_j \sr{t*}{\ra} \tilde u_j$,  
           $v_j \sr{t*}{\ra} \tilde v_j$, 
           $v_j^{-1} \sr{t*}{\ra} \tilde v_j'$, 
           $z_j =_{G_j} w_jy_j$ ,
           $z_j' =_{G_j} y_j'w_j^{-1}$, and
           $l(y_j)+l(y_j') \le l(z_j)+l(z_j') \le l(y_j)+l(y_j')+2l(w_j).$ 
\end{description}
Then for each $1 \le j \le n$
the word $x \in X^*$ also satisfies either:
\begin{description}

\item[(a')] In the case that $\pi_j(y) =y_j \in \geocl_j$: \\ 
           $\pi_j(x)=\hat v_j x_j \hat v_j'$  for some\\           
           $\hat v_j \in (\geol_j\$)^*$,
           $\hat v_j' \in (\$ \geol_j)^*$, 
           and  $x_j \in X_j^*$ such that \\
           $v_j \sr{t*}{\ra} \hat v_j$, $v_j^{-1} \sr{t*}{\ra} \hat v_j'$, 
           $x_j =_{G_j} w_jy_{j}w_j^{-1}$ and    \\
           $l(y_j) \le l(x_j) \le l(y_j)+2l(w_j)$,
\hspace{.1in} or
\item[(b')] In the case that $\pi_j(y) \in \widetilde U_j$:  \\
           $\pi_j(x)=\hat v_jx_j\hat u_jx_j'\hat v_j'$ for some \\
           $\hat u_j \in \$(\geol_j\$)^*$,
           $\hat v_j \in (\geol_j\$)^*$,
           $\hat v_j' \in (\$ \geol_j)^*$, 
           and $x_j,x_j' \in X_j^*$ such that \\
           $u_j \sr{t*}{\ra} \hat u_j$,  
           $v_j \sr{t*}{\ra} \hat v_j$, 
           $v_j^{-1} \sr{t*}{\ra} \hat v_j'$, 
           $x_j =_{G_j} w_jy_j$ ,
           $x_j' =_{G_j} y_j'w_j^{-1}$, and
           $l(y_j)+l(y_j') \le l(x_j)+l(x_j') \le l(y_j)+l(y_j')+2l(w_j)$.
\end{description}

\smallskip

{\narrower
\noindent {\it Proof of claim.} 
Since no length
reducing operation $z \sr{x0i}{\ra} x$ can be performed on a subword 
$a$ of
$z$ satisfying $\pi_i(a) \in \geol_i (\$ \geol_i)^*$,
the associated operation $\pi_i(z) \sr{s0i}{\ra} \pi_i(x)$
must apply to a subword of $\pi_i(z)$ disjoint from 
the $\tilde v_i$, $\tilde v_i^{-1}$, and $\tilde u_i$ subwords.

\noindent{\it Case 1.}  Suppose that $\pi_i(y) =y_i \in \geocl_i$.  Then 
the associated (s0$i$) operation must have the form
$\tilde v_i z_i \tilde v_i' \ra \hat v_i x_i \hat v_i'$ 
for some $x_i \in X_i^*$
where 
$\hat v_i = \tilde v_i$, $\hat v_i' = \tilde v_i'$,
$x_i \ei z_i \ei w_iy_{i}w_i^{-1}$, and
$l(x_i)<l(z_i) \le l(y_i)+2l(w_i)$.
Now since $x_i \sim_i y_i$ and the word
$y_i=\pi_i(y) \in \geocl_i$ is a \geocon\ for 
the group $G_i$ over the generating set $X_i$ in this case,  we must
have $l(y_i) \le l(x_i)$.  

If $x_i$ were the empty word $\la$, then since $x_i \sim_i y_i$
and $y_i \in \geocl_i$, we have $y_i = \la$ as well.
Then $w_iw_i^{-1} \sr{s0-1i*}{\ra} z_i \sr{s0i}{\ra} x_i$,
so we have $l(w_i)>0$.
Now $w \sr{x2*}{\ra} \tilde ww_i$ for a word $\tilde w \in X^*$, 
 $wyw^{-1}=_G \tilde wy\tilde w^{-1}$, and $l_X(w)>l_X(\tilde w)$.
But then replacing $w$ with $\tilde w$ contradicts our choice of words
$y \in \cap_{i=1}^n \pi_i^{-1}(\geocl_i \cup \widetilde U_i)$
and $w \in \geol$ with $l_X(\tht(wyw^{-1}))<l_X(y)$ 
and $l_X(y)+2l_X(w)$ minimal with respect to this property.
Hence $l(x_i) \ge 1$.  

Then the commutative diagrams above the Claim show that
for all $j \neq i$, $\pi_j(z) \sr{t*}{\ra} \pi_j(x)$.
Hence $v_j \sr{t*}{\ra} \tilde v_j \sr{t*}{\ra} \hat v_j$
and since $v_j, \tilde v_j \in (\geol_i \$)^*$, then
$\hat v_j \in (\geol_i \$)^*$.  
The proofs that $u_j \sr{t*}{\ra} \hat u_j \in \$ (\geol_i \$)^*$
and $v_j' \sr{t*}{\ra} \hat v_j' \in (\$ \geol_i)^*$
are identical.  Since the subwords $z_j$ and $z_j'$
of $\pi_j(z)$ lie in $X_j^*$, the operation $\pi_j(z) \sr{t*}{\ra} \pi_j(x)$
can't affect these subwords, and so $z_j=x_j$ and $z_j'=x_j'$
for all indices $j \neq i$.
Therefore for the word $x$,
conditions (a') or (b') hold
for all indices $j$, completing Case 1.  

\smallskip

\noindent{\it Case 2.}  Suppose that $\pi_i(y) \in \widetilde U_i$.  
Similar to the argument in the previous case, 
the (s0$i$) operation associated to the rewriting operation
$z \sr{x0i}{\ra} x$ has the form
$\pi_i(z)=\tilde v_iz_i\tilde u_iz_i'\tilde v_i' \ra 
\hat v_ix_i\hat u_ix_i'\hat v_i'=\pi_i(x)$ 
where $\tilde v_i=\hat v_i$, $\tilde u_i=\hat u_i$, 
$\hat v_i' = \tilde v_i'$,
and either $x_i \ei z_i$ with $l(x_i)<l(z_i)$ and $x_i'=z_i'$, 
or else $x_i=z_i$ and $x_i' \ei z_i'$ with $l(x_i')<l(z_i')$.
We consider the first of these two forms of the rewriting operation;
the proof in Case 2 for the second form is similar.
Now $x_i \ei z_i \ei w_iy_i$ and $x_i' = z_i' \ei y_i'w_i^{-1}$
with  $l(x_i)+l(x_i') < l(z_i) + l(z_i') \le l(y_i)+l(y_i')+2l(w_i)$.
Moreover, $x_i'x_i \ei y_i'w_i^{-1}w_iy_i \ei y_i'y_i$.
The fact that 
$y_iu_iy_i'=\pi_i(y) \in \widetilde U_i$ in this case implies that
$y_i'y_i \in \geol_i$, i.e., $y_i'y_i$ is a geodesic word.  
Hence $l(y_i)+l(y_i') 
\le l(x_i)+l(x_i')$.

If $x_i$ were the empty word, then since $x_i \ei w_iy_i$ and $w_i \in \geol_i$,
the word $w_i^{-1}$ is another geodesic representative of $y_i$,
and so $y_i'w_i^{-1}$ is also geodesic.  As in the
previous case, $w \sr{x2*}{\ra}w'w_i$ and
$y \sr{x2*}{\ra} y_iy''y_i'$.  But then replacing $w$
with $w'$ and $y$ with $y''y_i'w_i^{-1}$ 
again contradicts our choice of words
$y \in \cap_{i=1}^n \pi_i^{-1}(\geocl_i \cup \widetilde U_i)$
and $w \in \geol$ with $l_X(\tht(wyw^{-1}))<l_X(y)$ 
and $l_X(y)+2l_X(w)$ minimal with respect to this property.
Hence $l(x_i) \ge 1$.  Similarly $l(x_i') \ge 1$.

Then for all $j \neq i$, $\pi_j(z) \sr{t*}{\ra} \pi_j(x)$.
Therefore as in the prior case, for the word $x$,
condition (a') or (b') holds
for all indices $j$, completing Case 2 and the 
proof of the Claim.

}

\smallskip

Since the word $wyw^{-1}$ satisfies the property that
one of (a) or (b) holds for all $1 \le j \le n$, iteratively
applying this claim
shows that for each $j$, one of (a') or (b') must hold for
the normal form word $x:=\tht(wyw^{-1})$.
Denoting  the number of occurrences of $X_i$ letters in
a word $u$ over $X_i \cup \$$ by $l_i(u)$, then
whenever $u \sr{t*}{\ra} \tilde u$ we have
$l_i(u)=l_i(\tilde u)$.
Therefore $l_X(\tht(wyw^{-1})) 
= \sum_{i=1}^n l_i(\pi_i(\tht(wyw^{-1})))
\ge \sum_{i=1}^n l(y_{i})+l_i(u_i) +l(y_i') = l_X(y)$.
But our initial assumption on the pair $y,w$ included
the inequality $l_X(\tht(wyw^{-1})) < l_X(y)$, resulting in the
required contradiction.
\end{proof}

We are now ready to prove Theorem~\ref{thm:graphprodgcl}.
\begin{proof}
The class of regular languages is closed under finitely 
many operations of 
union, intersection, concatenation, and Kleene star (i.e.~$()^*$),
and is closed under inverse images of monoid
homomorphisms (see, for example,~\cite[Chapter~3]{hu}).

Proposition~\ref{prop:geoword} shows that 
$\geol=\cap_{i=1}^n \pi_i^{-1}(\geol_i(\$\geol_i)^*)$.
Then applying these closure properties
 yields a new proof of 
the result of Loeffler, Meier, and Worthington \cite[Theorem 1]{lmw}
that whenever the sets $\geol_i$ of geodesics for the
vertex groups $G_i$ are regular languages, then the
language $\geol$ 
of geodesics for the graph product group
$G$ over $X$ is also regular.

From Proposition~\ref{prop:geocon} we have
that the language of {\geocon}s for the
graph product group $G$ over the generating set 
$X$ satisfies $\geocl = \cap_{i=1}^n \pi_i^{-1}(\geocl_i \cup \widetilde U_i)$.
By hypothesis the language $\geocl_i$ is regular for each $i$, and so
the closure properties above imply that 
it suffices to
show that the language $\widetilde U_i$ over $X_i \cup \{\$\}$
is regular, given that the language
$\geol_i$ over the alphabet $X_i$ is regular.
The set
$\geol_i$ is recognized by a finite
state automaton with a set $Q=\{q_0, \dots q_m\}$
of $m+1$ states, where
$q_0$ is the start state, and with transition function
$\delta:Q \times X_i \rightarrow Q$.  Then $\geol_i$ can be written as 
$\geol_i=L_0\overline{L_0}\cup \cdots \cup L_{m}\overline{L_{m}}$, where 
for each $0 \le j \le m$, the set $L_j$ 
is the language of all the words $w$ over $X_i$ such that $\delta(q_0,w)=q_j$ 
and $\overline{L_j}$ 
is the set of all words $z$ such that $\delta(q_j,z)$ is an accept state. Then 
$L_j$ and $\overline{L_j}$ are regular languages.
Now $\widetilde U_i=\cup_{j=0}^m \overline{L_j}\$(\geol_i\$)^*L_j$, and 
therefore
$\widetilde U_i$ is also a regular language.
\end{proof}

\begin{example}\label{ex:racg}
{\em A}
right-angled Coxeter group 
{\em
is a graph product of cyclic groups $G_i = \langle a_i \mid a_i^2=1\rangle$
of order 2.  Then the graph product group $G=G\lm$
has generating set $X=\{a_i\}$, and 
there is a bijection between $T$ and the set of cliques of the graph $\lm$.
Theorem~\ref{thm:graphprodgcl} shows that the \wgeol\ $\geol(G,X)$
and the \wgeocl\ $\geocl(G,X)$ are both regular, and so
the \wgeos\ $\geos(G,X)$ and \wgeocs\ $\geocs(G,X)$ are both
rational functions.

In this example we provide the details for both spherical
and geodesic languages and series for
a specific right-angled Coxeter group, namely a graph product $G=G\lm$
of three groups $G_i = \langle a_i \mid a_i^2=1\rangle$ ($1 \le i \le 3$)
such that in the graph $\lm$ the vertex pairs $v_1,v_2$
and $v_2,v_3$ are adjacent, but the vertices $v_2$ and $v_3$
are not adjacent.  That is, 
$G=G_1 \times (G_2 * G_3)=\Z_2 \times (\Z_2 * \Z_2)$.
Here $X=\{a_1,a_2,a_3\}$ and $T=\{\aaa, \bb, \cc, \ab, \ac\}$.
The rewriting system
\[ \begin{array}{rlll}
 R := \{& \aaa^2 \sr{R0}{\ra} 1, & \aaa\ab  \sr{R0}{\ra} \bb, & 
                \hspace{-.6in}\aaa\ac \sr{R0}{\ra} \cc, \\
        & \bb^2 \sr{R0}{\ra} 1, & \bb\ab  \sr{R0}{\ra} \aaa, &  \\
        & \cc^2 \sr{R0}{\ra} 1, & \cc\ac  \sr{R0}{\ra} \aaa, &  \\
        & \ab\aaa \sr{R0}{\ra} \bb, & \ab\bb \sr{R0}{\ra} \aaa, & \\
        & \ab^2  \sr{R0}{\ra} \bb\bb, & \ab\ac \sr{R0}{\ra} \bb\cc, & \\
        & \ac\aaa \sr{R0}{\ra} \cc, & \ac\cc \sr{R0}{\ra} \aaa, & \\
        & \ac^2  \sr{R0}{\ra} \cc\cc, & \ac\ab \sr{R0}{\ra} \cc\bb, & \\
        & \aaa\bb \sr{R2}{\ra} \ab, & \aaa\cc \sr{R2}{\ra} \ac,  & \\
        & \bb\aaa \sr{R2}{\ra} \ab, & \bb\ac \sr{R2}{\ra} \ab\cc,  & \\
        & \cc\aaa \sr{R2}{\ra} \ac, & \cc\ab \sr{R2}{\ra} \ac\bb \hspace{.2in} \}  & 
\end{array} \]
is a subset of the set of all rewriting rules of types (R0)-(R2),
but is already sufficient to be a complete rewriting system for $G$
(see Proposition~\ref{prop:app}).  Then the set 
$irr(R) = T^* \setminus T^* L T^*$, where $L$ is the finite
set of words on the left hand sides of the rules in $R$,
is a regular language of normal forms for $G$ over $T$.  
Since the class of regular languages
is closed under images of monoid homomorphisms (\cite[Chapter~3]{hu}),
then the language $h(irr(R))$ is also regular.  Since the language
$H := h(irr(R))$ is a set of geodesic normal forms for $G$ over $X$,
then 
the strict growth series satisfies $f_H = \sphs(G,X)$;
hence the \wsphs\ $\sphs(G,X)$ is rational.
The combination of Theorem~\ref{thm:freeprodfinite} and
Proposition~\ref{prop:dirprodsph} shows that the
\wsphcl\ $\sphcl(G,X)$ is regular and the \wsphcs\ 
$\sphcs(G,X)$ is rational as well.

The homomorphisms $\pi_i: X^* \ra (X_i \cup \{\$\})^*$
are defined by 
$\pi_1(a_1)~=~a_1$,  $\pi_1(a_2) = \la$,  $\pi_1(a_3) = \la$, 
$\pi_2(a_1) = \la$,   $\pi_2(a_2) = a_2$, $ \pi_2(a_3) = \$ $,  
$\pi_3(a_1) = \la$,  $\pi_3(a_2) = \$ $, and $ \pi_3(a_3) = a_3~.$
Proposition~\ref{prop:geoword}
implies that the \wgeol\ 
$\geol:=\geol(G,X) = \cap_{i=1}^3 \pi_i^{-1} (\{\la,a_i\}(\$ \{\la,a_i\})^*)$.
Since the symbol $\$$ does not appear in the image of
the map $\pi_1$, then a geodesic word contains at most one occurrence
of the letter $a_1$.  Moreover, the preimage sets under
$\pi_2$ and $\pi_3$ imply that any two occurrences of $a_2$
must have an $a_3$ between them, and vice-versa.
Then
\[
\geol=\{\la,a_3\}(a_2a_3)^*\{\la,a_1\}(a_2a_3)^*\{\la,a_2\} \cup
        \{\la,a_2\}(a_3a_2)^*\{\la,a_1\}(a_3a_2)^*\{\la,a_3\}~.
\]
The strict growth function for this language satisfies
$\phi_\geol(0)=1$, $\phi_\geol(1)=3$, and $\phi_\geol(n)=2n+2$
for all $n \ge 2$.  The \wgeos\ for this group satisfies
$\geos(G,X)(z) = (1+z+z^2-z^3) / (1-z)^2$.

Applying Proposition~\ref{prop:geocon}, the \wgeocl\ is
$\geocl := \geocl(G,X) = 
\cap_{i=1}^3 \pi_i^{-1}(\{\la,a_i\} \cup \{\la,a_i\}(\$\{\la,a_i\})^*\$
\cup (\$\{\la,a_i\})^*\$a_i)$.  Analyzing this in the same way, then
\[\begin{array}{rl}
\geocl =& \{a_2,a_3,a_1a_2,a_1a_3,a_2a_1,a_3a_1\} \cup 
(a_2a_3)^*\{\la,a_1,a_2a_1a_3\}(a_2a_3)^*    \\
&  \cup (a_3a_2)^*\{\la,a_1,a_3a_1a_2\}(a_3a_2)^*~.
\end{array}\]
The strict growth function for this language satisfies
$\phi_\geocl(0)=1$, $\phi_\geocl(1)=3$, 
$\phi_\geocl(2)=6$, $\phi_\geocl(3)=6$,
$\phi_\geocl(2k)=2$,
and $\phi_\geocl(2k+1)=4k+2$ for all $k \ge 2$.  Then the
\wgeocs\ is given by
$\geocs(G,X)(z) = (1+3z+4z^2-9z^4+z^5+4z^6) / (1-z^2)^2$.
}
\end{example}


\begin{example}\label{ex:psl2}
{\em
Let $G$ be the projective special linear group 
$G:=PSL_2(\Z)=\Z_2 * \Z_3=\langle a,b,c \mid a^2=1,b^2=c, bc=1\rangle$ 
with the generating set $X=\{a,b,c\}$.  Theorem~\ref{thm:freeprodfinite}
shows that the  
\wsphcl\ $\sphcl(G,X)$ is not regular and the corresponding
\wsphcs\ is not rational.  However, from Theorem~\ref{thm:graphprodgcl}
we have that the \wgeocl\ $\geocl(G,X)$ is regular and 
the \wgeocs\ $\geocs(G,X)$ is a rational function.
Indeed, it follows from Theorem~\ref{thm:graphprodgcl}
that for any graph product of finite groups, the
\wgeocl\ is regular and the \wgeocs\ is rational, with respect to 
a union of finite generating sets of the vertex groups.
} 
\end{example}

\begin{theorem}\label{thm:dirfreegcs}
Let $G$ and $H$ be groups with finite inverse-closed generating
sets $A$ and $B$, respectively.  
Let $\geocs_G(z):=\geocs(G,A)(z)=\sum_{i=0}^\infty r_iz^i$ and 
$\geocs_H(z):=\geocs(H,B)(z)=\sum_{i=0}^\infty s_iz^i$ be the \wgeocs, 
and let $\geos_G:=\geos(G,A)$ and $\geos_H:=\geos(H,B)$
be the \wgeos\ for these pairs.

(i) The \wgeocs\ $\geocs_\times$ of the direct product 
$G \times H = \langle G,H \mid [G,H] \rangle$ of groups 
$G$ and $H$ with respect to the generating set $A \cup B$ is given by 
$\geocs_\times = \sum_{i=0}^\infty \delta_iz^i$
where $\delta_i:=\sum_{j=0}^i \binom{i}{j} r_j s_{i-j}$.

(ii) The \wgeocs\ $\geocs_*$ of the free product 
$G*H$ of the groups 
$G$ and $H$ with respect to the generating set $A \cup B$ is given by 
$$\geocs_* - 1 = (\geocs_G -1) + (\geocs_H - 1) -
  z \frac{d}{dz} \ln\left[ 1 - (\geos_G-1)(\geos_H-1) \right].$$
\end{theorem}

\begin{proof}
Denote the {\wgeol}s by 
$\geocl_G:=\geocl(G,A)$,
$\geocl_H:=\geocl(H,B)$,
$\geol_G:=\geol(G,A)$, and
$\geol_H:=\geol(H,B)$.

\smallskip

(i) The proof in this case follows the same argument as
the proof of the formula for the \wgeos\ of $G \times H$
in terms of the \wgeos\ of $G$ and $H$ in~\cite[Proposition~1]{lmw}.
In particular, in $G \times H$ each \geocon\ word $w$ of length $i$ can be
obtained by taking a word $y$ in $\geocl_G$ of length $0 \le j \le i$
and a word $z$ in $\geocl_H$ of length $i-j$, and ``shuffling'' the letters
so that the letters of $y$ and $z$ appear in the same order,
but not necessarily contiguously, in $w$.  

\smallskip

(ii) The \wgeocl\ $\geocl_*:=\geocl(G*H,A \cup B)$ can be written as
a disjoint union
$$\geocl_* = \{\la\} \cup (\geocl_G \setminus \{\la\}) \cup 
   (\geocl_H \setminus \{\la\}) \cup 
    \geocl_{A\bullet} \cup \geocl_{B\bullet}$$
where  $\la$ is the empty word,
$\geocl_{A\bullet}$ is the set of all \geocon\ words beginning
with a letter in $A$ and containing at least one letter in $B$, 
and similarly 
$\geocl_{B\bullet}$ is the set of all \geocon\ words beginning
with a letter in $B$ and containing at least one letter in $A$.
As a consequence, 
$$\geocs_* = 1+ (\geocs_G -1) + (\geocs_H -1) +
    f_{\geocl_{A\bullet}} + f_{\geocl_{B\bullet}},$$ 
where as usual $f_L$ denotes the strict growth series of the language $L$.
Now 
$\geocl_{A\bullet}$ can also be decomposed as a disjoint union
\begin{eqnarray*}
\geocl_{A\bullet} &=& \cup_{n=1}^\infty \{y_1z_1 \cdots y_nz_ny_{n+1} \mid 
  y_2,...,y_{n},y_{n+1}y_1 \in \geol_G \setminus \{\la\}, y_1 \in A^+, \\ 
& & \hspace{1.8in} z_1,...,z_n \in \geol_H \setminus \{\la\}\}. 
\end{eqnarray*}
As a consequence, the growth series of this language is 
\[ f_{\geocl_{A\bullet}}=\sum_{n=1}^\infty 
  (\geos_H-1)[(\geos_G-1)(\geos_H-1)]^{n-1}\left(z\frac{d}{dz} \geos_G\right), \]
where the $i$-th coefficient in the series 
$z\frac{d}{dz} \geos_G = \sum_{i=0}^\infty \alpha_i z^i$,
given by
\[ \alpha_i = i \cdot (\# \mbox{ of geodesics in } (G,A) \mbox{ of length } i),\]
counts the number of pairs of words $y_1,y_n$
with $y_1 \in A^+$, $y_n \in A^*$, and $y_ny_1$ a geodesic 
for the pair $(G,A)$ of length $i$.  The formula for
$f_{\geocl_{B\bullet}}$ is obtained in the same way.
Putting these together, then
\begin{eqnarray*}
\geocs_* &=& 1+ (\geocs_G -1) + (\geocs_H -1) +
 \sum_{n=1}^\infty (\geos_H-1)[(\geos_G-1)(\geos_H-1)]^{n-1}\left(z\frac{d}{dz} \geos_G\right) + \\
         & & \hspace{1.75in}
 \sum_{n=1}^\infty (\geos_G-1)[(\geos_H-1)(\geos_G-1)]^{n-1}\left(z\frac{d}{dz} \geos_H\right)\\
         &=& \geocs_G + \geocs_H - 1 +
 z\left(\frac{d}{dz} \geos_G\right)(\geos_H-1)\frac{1}{1-(\geos_G-1)(\geos_H-1)} + \\
         & & \hspace{1.75in}
 z\left(\frac{d}{dz} \geos_H\right)(\geos_G-1)\frac{1}{1-(\geos_G-1)(\geos_H-1)} \\
         &=& \geocs_G + \geocs_H - 1 + 
 z\frac{\left(\frac{d}{dz} \geos_G\right)(\geos_H-1) + \left(\frac{d}{dz} \geos_H\right)(\geos_G-1)}{1-(\geos_G-1)(\geos_H-1)}, 
\end{eqnarray*} 
resulting in the required formula.
\end{proof}

\begin{example}\label{ex:f2geocs}
{\em 
In the free group $F_2=\mathbb{Z} * \mathbb{Z}$, associate $G=\langle a \mid~ \rangle$ with 
the first copy of the integers and $H=\langle b \mid ~\rangle$ with the second,
and let $A=\{a,a^{-1}\}$, $B=\{b,b^{-1}\}$. 
We have 
\[
\geocs_G=\geocs_H=\geos_G=\geos_H=\frac{1+z}{1-z} ,
\]
and so
$1-(1-\geos_G)(1-\geos_H)=\frac{1-2z-3z^2}{(1-z)^2}$.
Plugging these into the formula in Theorem~\ref{thm:dirfreegcs}(ii),
we obtain 
\[\geocs(F_2,\{a^{\pm 1},b^{\pm 1}\})=\geocs_*=
\frac{1 + z - z^2 - 9z^3}{1 - 3z - z^2 + 3z^3}.\]
(We note that in \cite[Corollary 14.1]{rivin} Rivin has
computed the \weqcs\ for this group (see Section~\ref{sec:open} 
for the definition of the \weqcl\ and \weqcs), and that in
this example, the \weqcl\ and the \wgeocl\ are the same set,
and so the corresponding growth series are also equal.
The rational function above
differs from Rivin's formula
by adding 1, because Rivin's growth series
does not count the constant term corresponding to the empty word
$\la$ in the \weqcl.)
}
\end{example}

\begin{example}\label{ex:psl2again}
{\em
Let $G$ and $H$ be finite groups with generating sets $A := G \setminus \{1_G\}$
and $B := H \setminus \{1_H\}$, and let $m := |A| = |G| - 1$ and $n := |B| = |H| -1$.
The corresponding languages
are $\geocl_G=\geol(G,Y)=A \cup \{\la\}$ and 
$\geocl_H=\geol(H,B)=Z \cup \{\la\}$, where $\la$ is the empty word, and
hence we have growth series
$\geocs_G=\geos(G,A)=mz+1$ and 
$\geocs_H=\geos(H,B)=nz+1$.
For the free product $G*H$, with generating set $A \cup B$,
the formula in Theorem~\ref{thm:dirfreegcs}
shows that the \wgeocs\ is
\[ 
\geocs_*(z)=\geocs(G*H,A \cup B)(z)=\frac{1+(m+n)z+mnz^2-mn(m+n)z^3}{1-mnz^2} .
\]
In particular, for the group $P:=PSL_2(\Z)$ with the generating
set $X$ from Example~\ref{ex:psl2}, the \wgeocs\ is given
by the rational function
\[
\geocs(P,X)(z)=\frac{1+3z+2z^2-6z^3}{1-2z^2} .
\]
}
\end{example}

\begin{proposition}\label{prop:amalgprodfinite}
If $G$ and $H$ are finite groups with a common subgroup $K$,
then the free product $G *_K H$ of $G$ and $H$ amalgamated over 
$K$, with respect 
to the generating set $X := G \cup H \cup K - \{1\}$,
has regular \wgeocl\ $\geocl(G *_K H,X)$
and rational \wgeocs\ $\geocs(G*_KH,X)$.
\end{proposition}

\begin{proof}
Let $X_G:=G \setminus K$, $X_H:=H \setminus K$, and
$X_K:=K \setminus \{1\}$; then $X=X_K \cup X_G \cup X_H$.
Given a sequence $x_1,...,x_n$ of elements of $X$,
this sequence is called {\em cyclically reduced} if
either $n=1$ and $x_1 \in X$ 
or else $n >1$ and 
for each $1 \le i \le n$ the elements $x_i$ and $x_{i+1(\text{mod } n)}$
lie in $X_G \cup X_H$ and 
are from different factors (i.e. if $x_i \in X_G$ 
then $x_{i+1} \in X_H$
and vice versa), and the product $x_1 \cdots x_n$ in $G *_K H$ 
is the associated {\em cyclically reduced product}.
Every element of $G *_K H$ is conjugate to
a cyclically reduced product.  Moreover,
by~\cite[Theorem~IV.2.8]{lyndonschupp}, given any cyclically
reduced sequence $x_1,...,x_n$ with $n \ge 2$ and any $g \in G *_K H$, 
every product of a cyclically reduced sequence $x_1',...,x_k'$ satisfying
$gx_1 \cdots x_kg^{-1} =_{G *_K H} x_1' \cdots x_k'$ can be
obtained by cyclically permuting the original sequence $x_1,...,x_n$
and conjugating the resulting product by an element of $K$.
Now every \geocon \ word over $X$ must be 
a cyclically reduced product, and this theorem implies
also that every cyclically reduced product is a \geocon.
That is, 
\[\geocl(G *_K H,X)=\{\lambda\} \cup X \cup (X_GX_H)^* \cup (X_GX_H)^*,\]
giving a regular expression for the language $\geocl(G *_K H,X)$.
\end{proof}

\section{Open questions}\label{sec:open}

We remark that intermediate between the two {\conjl}s 
defined in Section~\ref{sec:intro}
is a third set of words given by the subset of $\sphl(G,X)$
defined by
\[\eqcl=\eqcl(G,X):=\{y_g \mid g \in G, |g|=|g|_c\}
~,\]
which we refer to as the {\em \weqcl}
for $G$ over $X$.  
It is immediate from the
definitions that $\sphcl \subseteq \eqcl \subseteq \geocl$.
We denote the strict growth series for this language by
\[\eqcs=\eqcs(G,X):=f_{\eqcl(G,X)},\]
called the {\weqcs}.
In~\cite[Corollary~14.1]{rivin} 
Rivin gives a rational function
formula for the
\weqcs\ (which he denotes by $\mathcal{F}[C_{F_k}](z)$) 
for a finitely generated free group with respect to
a free basis (see also ~Example~\ref{ex:f2geocs} for
the rank 2 case), and so in this case the free group does not
give an obstruction to rationality being preserved by free products.  

\begin{question}
If $G$ is a graph product of finitely many
groups $G_i$ and each group $G_i$
has a finite inverse-closed generating set $X_i$ such
that $\eqcl(G_i,X_i)$ is a regular language,
is the language $\eqcl(G,\cup_i X_i)$ regular?
Is rationality of the \weqcs\ $\eqcs(G,X)$ preserved by
direct and free products? 
\end{question}

In some cases, regularity of languages and 
rationality of growth series
associated to groups are known to depend upon
the generating set chosen.  For example,
Stoll~\cite{stoll} has shown that rationality of the
usual (cumulative) growth series $b_{\sphl(G,X)}$
(and hence also of the \wsphs\ $\sphs(G,X)=f_{\sphl(G,X)}$)
depends upon the generating set for the
higher Heisenberg groups, and Cannon~\cite[p.~268]{neushap}
has shown that regularity of the \wgeol\ $\geol(\Z^2 \rtimes \Z_2,X)$
depends
on the generating set $X$ for a semidirect product of
$\Z^2$ by the cyclic group of order 2.
In the case of conjugacy growth, Hull and 
Osin~\cite[Theorem~1.3]{hullosin} have shown 
an example of a finitely generated group $G$
with a finite index subgroup $H$
such that the conjugacy growth function $\beta_{\sphcl(G,X)}$
grows exponentially, but $H$ has only two conjugacy classes.
Then the spherical and \wgeocs\ for $G$ are both infinite
series, but these two series for $H$ are both polynomials.

\begin{question}
Does there exist a finite inverse-closed generating set $X$ for the free 
group $F$ on two generators 
such that $\sphcl(F,X)$ is regular?
\end{question}

For the free basis which gives a non-regular \wsphcl\ 
(Proposition~\ref{prop:freegpscl}), we also consider
formal language theoretic classes that are less 
restrictive than context-free languages.

\begin{question}
Let $F=F(a,b)$ be the free group on generators $a$ and $b$. Is
$\sphcl(F,\{a^{\pm 1},b^{\pm 1}\})$  an indexed language?
A context-sensitive language?
\end{question}

In Corollary~\ref{cor:freesubscl} we show that the
\wsphcl\ cannot be regular (or indeed context-free)
in any group that contains
a free subgroup as a direct or free factor, with respect to
a generating set that is the union of the free basis and
the generators of the other factor.  
This leads us to wonder whether free groups are ``poison
subgroups'' from the viewpoint of regular {\wsphcl}s.

\begin{question}
Let $G$ be a group with $F=F(a,b)$ as subgroup, and let
$X$ be an inverse-closed generating set for $G$.
Can $\sphcl(G,X)$ be a regular language?
\end{question}

As we remarked in Section~\ref{sec:intro}, 
Cannon has shown 
that for word hyperbolic groups, the \wgeol\ 
for every finite generating set is 
regular~\cite[Chapter~3]{echlpt}.  

\begin{question}
Let $G$ be a word hyperbolic group and let $X$
be a finite generating set for $G$.  Is the
\wgeocl\ $\geocl(G,X)$ necessarily regular?
\end{question}

\section*{Acknowledgments}

The authors were partially supported by the Marie Curie Reintegration Grant 230889.
The authors thank Mark Brittenham for helpful discussions.


\bigskip

\textsc{L. Ciobanu,
Mathematics Department,
University of Fribourg,
Chemin du Muse\'e 23,
CH-1700 Fribourg, Switzerland
}

\emph{E-mail address}{:\;\;}\texttt{laura.ciobanu@unifr.ch}
\medskip

\bigskip

\textsc{S. Hermiller,
Department of Mathematics,
University of Nebraska,
Lincoln, NE 68588-0130 USA 
}

\emph{E-mail address}{:\;\;}\texttt{smh@math.unl.edu}
\medskip

\vfill\eject

\appendix
\section{Knuth-Bendix algorithm for graph products}\label{app}

The goal of this appendix is to prove the following theorem.

\begin{theorem}\label{thm:kb}
A graph product of finitely many groups admits a 
weightlex complete rewriting system that is compatible
with shortlex complete rewriting systems on the
vertex groups.
\end{theorem}

Let $\lm$ be a finite simplicial graph with $n$ vertices $v_1,...,v_n$
and suppose that for each $1 \le i \le n$ the vertex $v_i$ is labeled by
a group $G_i$ that has a finite inverse-closed generating set $X_i$.
For two words $u,w \in X_i^*$, we write
$u \ei w$ if $u$ and $w$ represent the same element of $G_i$. 
Let $G$ be the associated graph product
with generating set $X:=\cup_{i=1}^n X_i$.  For
words $y,z \in X^*$, write $y=_Gz$ if $y$ and $z$ represent
the same element of $G$.

As in Section~\ref{sec:geodesic}, 
given any subset $t$ of $X$, we define the {\it support}
$supp(t)$ of $t$ to be the set of all vertices $v_i$ of $\lm$ such that
$t$ contains an element of $X_i$. Let 
$T$ be the set of all nonempty subsets $t$ of $X$ satisfying the
properties that $supp(t)$ is a clique of $\lm$, and for each $v_i \in supp(t)$,
the intersection $t \cap X_i$ is a single element of $X_i$.
Each $t=\{a_1,...,a_k\} \in T$ determines a
well-defined element 
$a_1 \cdots a_k$ of $G$, and each $a \in X$
is the element of $G$ associated to $\{a\} \in T$.
By slight abuse of notation, we will consider
$X \subseteq T \subseteq G$, so that $T$ is another
inverse-closed generating set for $G$.

For each index $i$ fix a total ordering on 
$X_i$, and let $<_i$
denote the corresponding shortlex ordering on $X_i^*$.
To ease notation, we
let the empty set $\{ \}$ denote the empty word over $T$.
Whenever $t \in T \cup \{\emptyset\}$, $a \in X$, $t \cup \{a\} \in T$,
and $t \cap \{a\} = \emptyset$, 
we let $\{a,t\}$ denote the set $t \cup \{a\}$.

For each $1 \le i \le n$, let $(X_i,S_i)$ be a complete
rewriting system for $G_i$ that is compatible with
respect to the shortlex ordering $>_i$.  That is, 
$S_i \subseteq X_i^* \times X_i^*$ satisfies
the properties that for each $(u,v) \in S_i$
the inequality $u >_i v$ holds; 
the rewriting rules $xuy \ra_i xvy$ whenever $x,y \in X_i^*$
and $(u,v) \in S_i$ are {\it confluent}, in that
whenever
a word $w$ rewrites two words $w \ra_i w'$ and $w \ra_i w''$
using these rules, then there is a word $w'''$ such
that $w' \sr{*}{\ra_i} w'''$ and $w' \sr{*}{\ra_i} w'''$;
and the group $G_i$ is presented as a monoid by 
$Mon \langle X_i \mid \{u=v \mid (u,v) \in S_i \}\rangle$.

We define two sets of rewriting rules on words over $T$
as follows.  

\begin{description}

\item[(R1)] $\{a_1,t_1\} \cdots \{a_m,t_m\} \ra
\{b_1,t_1\} \cdots \{b_{k},t_{k}\} t_{k+1} \cdots t_m$ whenever
there exists an index $1 \le i \le n$ such that for all $1 \le j \le m$,
each $a_j,b_j \in X_i$, $t_j \in T \cup \{\emptyset\}$,
$\{a_j,t_j\} \in T$; and
$a_1 \cdots a_m \ra_i b_1 \cdots b_k$.

\item[(R2)] $t\{a,t'\} \ra \{a,t\}t'$ whenever
$a \in X$, $t \in T$, $t' \in T \cup \{\emptyset\}$,
and $\{a,t\}, \{a,t'\} \in T$.
\end{description}
(In the case that for each $1 \le i \le n$ the complete
rewriting system $S_i$ includes all rules
$w \ra_i z$ for all $w,z \in X_i^*$ satisfying
$w \ei z$ and $w >_i z$, then the rules (R1) above
are exactly the rules (R0),(R1) defined in Section~\ref{sec:geodesic}.)

Whenever $1 \le j \le 2$, $\alpha \subseteq \{1,2\}$ and $w,x \in T^*$,
we write $w \sr{Rj}{\ra} x$ if $w$ rewrites to $x$ via
exactly one application of rule (R$j$), and
$w \sr{R\alpha*}{\ra} x$ if $x$ can be obtained
from $w$ by a finite (possibly zero) number of rewritings
using rules of the type (R$j$) for $j \in \alpha$.

Let $R$ denote the set of all rewriting rules of the
form (R1) and (R2).  Note that the generating
set $T$ of $G$ together with the relations given by
the rules of $R$ form a monoid presentation of $G$;
hence, $(T,R)$ is a {\it rewriting system} for the
graph product group $G$.  

\begin{proposition}\label{prop:app}
The rewriting system $(T,R)$ for $G$ is {\em complete}.
That is, $R$ is {\em terminating} (no word $w \in T^*$
can be rewritten infinitely many times)
and $R$ is {\em confluent} (whenever
a word $w$ rewrites two words $w \ra w'$ and $w \ra w''$
using these rules, then there is a word $w'''$ such
that $w' \sr{R0-2*}{\ra} w'''$ and $w' \sr{R0-2*}{\ra} w'''$).
\end{proposition}

\begin{proof}
Define a partial ordering on $T$ by $\{a,t\} > \{b,t\}$ whenever
$a,b \in X_i$ for some $i$, $a >_i b$, $t \in T \cup \emptyset$,
and $\{a,t\} \in T$; and by
$\{a,t\} < t$ whenever $a \in X$ and $t, \{a,t\} \in T$.
For each $t$ in $T$, define the {\em weight} $wt(t)$
of $t$ to be the number of elements of $t$ as a
subset of $X$ (equivalently, $wt(t)$ is the
number of vertices in $supp(t)$).
Then all of the
rules in the rewriting system $R$ decrease the associated
weightlex ordering on $T^*$.  Since the weightlex ordering
is compatible with concatenation and well-founded,
no word $w \in T^*$
can be rewritten infinitely many times; that is, the
system $R$ is {\it terminating}.

Next we check via the Knuth-Bendix algorithm
that this system is also {\it confluent}.
There are two types of overlapping rules, or
{\em critical pairs}, we need to check:
\begin{description}
\item[External] $rs \ra u, st \ra v \in R$ with $s \ne \lambda$.
\item[Internal] $rst \ra u, s \ra v \in R$.
\end{description}

We denote a critical pair between an (R$k$) rule and an
(R$l$) rule ($1 \le k,l \le 2$)
that is of external type by $kl$ext, etc.
For any rewriting rule
$a_1 \cdots a_m \ra_i b_1 \cdots b_k$ in $S_i$
with each $a_j,b_j \in X_i^*$  and $k<m$, we let
$b_{k+1} = \cdots = b_{m} = \emptyset$.
Checking critical pairs:

\bigskip

\noindent {\bf 22ext:}

\noindent If $t' \neq \emptyset$:

$\uu{t\{a,t'\}}\{b,t''\} \sr{R2}{\ra} \{a,t\}\uu{t'\{b,t''\}} \sr{R2}{\ra} \{a,t\}\{b,t'\}t''$
\hspace{.15in} and

$t\uu{\{a,t'\}\{b,t''\}} \sr{R2}{\ra} \uu{t\{a,b,t'\}}t'' \sr{R2}{\ra} \{a,t\}\{b,t'\}t''$.

\bigskip

\noindent If $t' = \emptyset$:

$\uu{t\{a\}}\{b,t''\} \sr{R2}{\ra} \{a,t\}\{b,t''\}$
\hspace{.15in} and

$t\uu{\{a\}\{b,t''\}} \sr{R2}{\ra} \uu{t\{a,b\}}t'' \sr{R2}{\ra} \{a,t\}\uu{bt''} \sr{R2*}{\ra} 
     \{a,t\}\{b,t''\}$.

\bigskip

\noindent {\bf 22int:}

$\uu{t\{a,b,t'\}} \sr{R2}{\ra} \uu{\{a,t\}\{b,t'\}} \sr{R2}{\ra} \{a,b,t\}t'$ \hspace{.15in} and

$\uu{t\{a,b,t'\}} \sr{R2}{\ra} \uu{\{b,t\}\{a,t'\}} \sr{R2}{\ra} \{a,b,t\}t'$. 

\bigskip

\noindent {\bf 12ext:}

\noindent If $\{b_m,t_m\} \neq \emptyset$: 

$\uu{\{a_1,t_1\} \cdots \{a_m,t_m\}}\{c,t'\} \sr{R1}{\ra} 
\{b_1,t_1\} \cdots \uu{\{b_m,t_m\}\{c,t'\}} \sr{R2}{\ra}
\{b_1,t_1\} \cdots \{c,b_m,t_m\}t'$ 

and

$\{a_1,t_1\} \cdots \uu{\{a_m,t_m\}\{c,t'\}} \sr{R2}{\ra} 
\uu{\{a_1,t_1\} \cdots \{c,a_m,t_m\}}t' \sr{R1}{\ra}
\{b_1,t_1\} \cdots \{c,b_m,t_m\}t'$.

\bigskip

\noindent If $\{b_m,t_m\} = \emptyset$:

$\uu{\{a_1,t_1\} \cdots \{a_m,t_m\}}\{c,t'\} \sr{R1}{\ra} 
\{b_1,t_1\} \cdots \{b_{m-1},t_{m-1}\}\{c,t'\}$

and

$\{a_1,t_1\} \cdots \uu{\{a_m,t_m\}\{c,t'\}} \sr{R2}{\ra} 
\uu{\{a_1,t_1\} \cdots \{c,a_m,t_m\}}t' \sr{R1}{\ra}$

$~$ \hspace{.15in} $\{b_1,t_1\} \cdots \{b_{m-1},t_{m-1}\}\uu{ct'} \sr{R2*}{\ra}
\{b_1,t_1\} \cdots \{b_{m-1},t_{m-1}\}\{c,t'\}$.

\bigskip

\noindent {\bf 21ext option a:}

\noindent If $t_1,...,t_{m-1} \neq \lambda$:

$\uu{t'\{a_1,t_1\}} \cdots \{a_m,t_m\} \sr{R2}{\ra} 
\{a_1,t'\}\uu{t_1\{a_2,t_2\} \cdots \{a_m,t_m\}} \sr{R2*}{\ra}$

$~$ \hspace{.15in} $\uu{\{a_1,t'\}\{a_2,t_1\} \cdots \{a_m,t_{m-1}\}}t_m  \sr{R1}{\ra}
\{b_1,t'\}\{b_2,t_1\} \cdots \{b_m,t_{m-1}\}t_m$ 

and

$t'\uu{\{a_1,t_1\} \cdots \{a_m,t_m\}} \sr{R1}{\ra} 
\uu{t'\{b_1,t_1\} \cdots \{b_m,t_m\}} \sr{R2*}{\ra}$

$~$ \hspace{.15in} $\{b_1,t'\}\{b_2,t_1\} \cdots \{b_m,t_{m-1}\}t_m$.

\bigskip

\noindent If $t_1,...,t_{i-1} \neq \lambda$ and $t_i=\lambda$:

$\uu{t'\{a_1,t_1\}} \cdots \{a_m,t_m\} \sr{R2}{\ra} $

$~$ \hspace{.15in} $\{a_1,t'\}\uu{t_1\{a_2,t_2\} \cdots \{a_m,t_m\}} \sr{R2*}{\ra}$

$~$ \hspace{.15in} $\uu{\{a_1,t'\}\{a_2,t_1\} \cdots \{a_{i},t_{i-1}\}\{a_{i+1},t_{i+1}\} 
    \cdots \{a_m,t_m\}}     \sr{R1}{\ra}$
    
$~$ \hspace{.15in} $\{b_1,t'\}\{b_2,t_1\} \cdots \{b_{i},t_{i-1}\}\{b_{i+1},t_{i+1}\} 
          \cdots \{b_m,t_m\}$ 
          
and

$t'\uu{\{a_1,t_1\} \cdots \{a_m,t_m\}} \sr{R1}{\ra}$ 

$~$ \hspace{.15in} $\uu{t'\{b_1,t_1\} \cdots \{b_{i-1},t_{i-1}\}} \{b_{i},t_{i}\} 
             \cdots \{b_m,t_m\} \sr{R2*}{\ra}$

$~$ \hspace{.15in} $\{b_1,t'\}\{b_2,t_1\} \cdots \{b_{i},t_{i-1}\}{\color{green}t_i}
      \{b_{i+1},t_{i+1}\} \cdots \{b_m,t_m\} =$

$~$ \hspace{.15in} $\{b_1,t'\}\{b_2,t_1\} \cdots \{b_{i},t_{i-1}\}\{b_{i+1},t_{i+1}\} 
       \cdots \{b_m,t_m\}$.

\eject

\noindent {\bf 21ext option b:}

$\uu{t'\{c,a_1,t_1\}} \cdots \{a_m,t_m\} \sr{R2}{\ra} 
\{c,t'\}\uu{\{a_1,t_1\} \cdots \{a_m,t_m\}} \sr{R1}{\ra}
\{c,t'\}\{b_1,t_1\} \cdots \{b_m,t_m\}$ 

and

$t'\uu{\{c,a_1,t_1\} \cdots \{a_m,t_m\}} \sr{R1}{\ra} 
\uu{t'\{c,b_1,t_1\}} \cdots \{b_m,t_m\}  \sr{R2}{\ra}
\{c,t'\}\{b_1,t_1\} \cdots \{b_m,t_m\}$.

\bigskip

\noindent {\bf 21int:}

\noindent If $\{b_{i-1},t_{i-1}\} \neq \emptyset$:

$\uu{\{a_1,t_1\} \cdots \{a_{i-1},t_{i-1}\}\{c,a_i,t_i\} \cdots \{a_m,t_m\}} \sr{R1}{\ra}$

$~$ \hspace{.15in} $\{b_1,t_1\} \cdots \uu{\{b_{i-1},t_{i-1}\}\{c,b_i,t_i\}} \cdots
  \{b_m,t_m\} \sr{R2}{\ra}$
  
$~$ \hspace{.15in}   
$\{b_1,t_1\} \cdots \{c,b_{i-1},t_{i-1}\}\{b_i,t_i\} \cdots \{b_m,t_m\}$ 

and

$\{a_1,t_1\} \cdots \uu{\{a_{i-1},t_{i-1}\}\{c,a_i,t_i\}} \cdots \{a_m,t_m\} \sr{R2}{\ra}$

$~$ \hspace{.15in}
$\uu{\{a_1,t_1\} \cdots \{c,a_{i-1},t_{i-1}\}\{a_i,t_i\} \cdots \{a_m,t_m\}} \sr{R1}{\ra}$

$~$ \hspace{.15in}
$\{b_1,t_1\} \cdots \{c,b_{i-1},t_{i-1}\}\{b_i,t_i\} \cdots \{b_m,t_m\}$.

\bigskip

\noindent If $\{b_{i-1},t_{i-1}\} = \emptyset$:

$\uu{\{a_1,t_1\} \cdots \{a_{i-1},t_{i-1}\}\{c,a_i,t_i\} \cdots \{a_m,t_m\}} \sr{R1}{\ra}$

$~$ \hspace{.15in}
$\{b_1,t_1\} \cdots \{b_{i-2},t_{i-2}\}\{c,t_i\} \cdots t_m$ 

and

$\{a_1,t_1\} \cdots \uu{\{a_{i-1},t_{i-1}\}\{c,a_i,t_i\}} \cdots \{a_m,t_m\} \sr{R2}{\ra}$

$~$ \hspace{.15in}
$\uu{\{a_1,t_1\} \cdots \{c,a_{i-1},t_{i-1}\}\{a_i,t_i\} \cdots \{a_m,t_m\}} \sr{R1}{\ra}$

$~$ \hspace{.15in}
$\{b_1,t_1\} \cdots \{b_{i-2},t_{i-2}\} \uu{c t_{i}}t_{i+1} \cdots t_m \sr{R2*}{\ra}$

$~$ \hspace{.15in}
$\{b_1,t_1\} \cdots \{b_{i-2},t_{i-2}\}\{c,t_i\}t_{i+1} \cdots t_m$.

\bigskip

\noindent {\bf 11ext option a:}

$\uu{\{a_1,t_1\} \cdots \{a_{i},t_{i}\}\{a_{i+1},c_1,t_1'\} \cdots \{a_m,c_{m-i},t_m'\}} 
    \{c_{m-i+1},t_{m-i+1}'\} \cdots \{c_l,t_l'\}        \sr{R1}{\ra}$
    
$~$ \hspace{.15in}    
$\{b_1,t_1\} \cdots \{b_{i},t_{i}\}\uu{\{b_{i+1},c_1,t_1'\} \cdots \{b_m,c_{m-i},t_m'\} 
    \{c_{m-i+1},t_{m-i+1}'\} \cdots \{c_l,t_l'\}}        \sr{R1}{\ra}$
    
$~$ \hspace{.15in}    
$\{b_1,t_1\} \cdots \{b_{i},t_{i}\}\{b_{i+1},d_1,t_1'\} \cdots \{b_m,d_{m-i},t_m'\} 
    \{d_{m-i+1},t_{m-i+1}'\} \cdots \{d_l,t_l'\}$      

and

$\{a_1,t_1\} \cdots \{a_{i},t_{i}\}\uu{\{a_{i+1},c_1,t_1'\} \cdots \{a_m,c_{m-i},t_m'\} 
    \{c_{m-i+1},t_{m-i+1}'\} \cdots \{c_l,t_l'\}}        \sr{R1}{\ra}$
    
$~$ \hspace{.15in}    
$\uu{\{a_1,t_1\} \cdots \{a_{i},t_{i}\}\{a_{i+1},d_1,t_1'\} \cdots \{a_m,d_{m-i},t_m'\}} 
    \{d_{m-i+1},t_{m-i+1}'\} \cdots \{d_l,t_l'\}        \sr{R1}{\ra}$
    
$~$ \hspace{.15in}    
$\{b_1,t_1\} \cdots \{b_{i},t_{i}\}\{b_{i+1},d_1,t_1'\} \cdots \{b_m,d_{m-i},t_m'\} 
    \{d_{m-i+1},t_{m-i+1}'\} \cdots \{d_l,t_l'\}$.

\eject

\noindent {\bf 11ext option b:}

\noindent For critical pair $a_1 \cdots a_m \ra_i b_1 \cdots b_k$
and $a_p \cdots a_q \ra_i c_p \cdots c_r$ of $S_i$ (where $p \le m \le q$,
$k \le m$, and $p-1 \le r \le q$):
There are rewritings $b_1 \cdots b_k a_{m+1} \cdots a_q \sr{*}{\ra_i} d_1 \cdots d_s$
and $a_1 \cdots a_{p-1} c_p \dots c_r \sr{*}{\ra_i} d_1 \cdots d_s$.  Then:

\bigskip

\noindent If $t_l \neq \emptyset$ for all $k+1 \le l \le m$:

$\uu{\{a_1,t_1\} \cdots \{a_m, t_m\}} \{a_{m+1}, t_{m+1}\}\cdots \{a_q, t_q\}   \sr{R1}{\ra}$
   
$~$ \hspace{.15in}   
$\{b_1,t_1\} \cdots \{b_k,t_k\} \uu{t_{k+1} \cdots t_{m} \{a_{m+1},t_{m+1}\} \cdots  \{a_q, t_q\}}
   \sr{R2*}{\ra}$
   
$~$ \hspace{.15in}   
$\uu{\{b_1,t_1\} \cdots \{b_k,t_k\} \{a_{m+1},t_{k+1}\} \cdots \{a_{q},t_{k+(q-m)}\}} 
     t_{k+q-m+1} \cdots  t_q \sr{R1*}{\ra}$

$~$ \hspace{.15in}
$\{d_1,t_1\} \cdots \{d_s,t_s\} t_{s+1} \cdots  t_q$

and

$\{a_1,t_1\} \cdots \{a_{p-1}, t_{p-1}\} \uu{\{a_p, t_p\}  \cdots \{a_q, t_q\}}   \sr{R1}{\ra}$

$~$ \hspace{.15in}   
$\uu{\{a_1,t_1\} \cdots \{a_{p-1},t_{p-1}\} \{c_p,t_p\} \cdots \{c_r,t_r\}} 
    t_{r+1} \cdots t_q \sr{R1*}{\ra}$

$~$ \hspace{.15in}
$\{d_1,t_1\} \cdots \{d_s,t_s\} t_{s+1} \cdots  t_q$.  

\bigskip

\noindent If $t_{k+1} \cdots t_m = t_{k+1}' \cdots t_u'$ for some 
$u \le m-1$ with each $t_l' \in T$:

\noindent Let $t_l'':=t_l$ for $1 \le l \le k$, $t_{l}'':=t_{l}'$
for $k+1 \le l \le u$, and $t_l'':=t_{l+m-u}$ for $u+1 \le l \le q-(m-u)$.

$\uu{\{a_1,t_1\} \cdots \{a_m, t_m\}} \{a_{m+1}, t_{m+1}\} \cdots \{a_q, t_q\}   \sr{R1}{\ra}$

$~$ \hspace{.15in}
$\{b_1,t_1\} \cdots \{b_k,t_k\} t_{k+1} \cdots t_{m} \{a_{m+1},t_{m+1}\} \cdots  \{a_q, t_q\} =$
   
$~$ \hspace{.15in}   
$\{b_1,t_1''\} \cdots \{b_k,t_k''\} \uu{t_{k+1}'' \cdots t_{u}'' 
  \{a_{m+1},t_{u+1}''\} \cdots  \{a_q, t_{q-(m-u)}''\}}   
     \sr{R2*}{\ra}$
   
$~$ \hspace{.15in}   
$\uu{\{b_1,t_1''\} \cdots \{b_k,t_k''\} \{a_{m+1},t_{k+1}''\} \cdots \{a_{q},t_{q-(m-k)}''\}} 
     t_{q-(m-k)+1}'' \cdots  t_{q-(m-u)}'' \sr{R1*}{\ra}$

$~$ \hspace{.15in}
$\{d_1,t_1''\} \cdots \{d_s,t_s''\} t_{s+1}'' \cdots  t_{q-(m-u)}''$

and

$\{a_1,t_1\} \cdots \{a_{p-1}, t_{p-1}\} \uu{\{a_p, t_p\} \cdots \{a_q, t_q\}}   \sr{R1}{\ra}$

$~$ \hspace{.15in}
$\uu{\{a_1,t_1\} \cdots \{a_{p-1},t_{p-1}\} \{c_p,t_p\} \cdots \{c_r,t_r\}} 
    t_{r+1} \cdots t_q \sr{R1*}{\ra}$

$~$ \hspace{.15in}
$\uu{\{d_1,t_1\} \cdots \{d_s,t_s\} t_{s+1} \cdots  t_q} \sr{R2*}{\ra}
\{d_1,t_1''\} \cdots \{d_s,t_s''\} t_{s+1}'' \cdots  t_{q-(m-u)}''$.

\bigskip

\noindent {\bf 11int option a:}

$\uu{\{a_1,t_1\} \cdots \{a_{i},t_{i}\}\{a_{i+1},c_1,t_{i+1}\} \cdots \{a_k,c_{k-i},t_k\} 
    \{a_{k+1},t_{k+1}\} \cdots \{a_m,t_m\}}        \sr{R1}{\ra}$
    
$~$ \hspace{.15in}    
$\{b_1,t_1\} \cdots \{b_{i},t_{i}\}\uu{\{b_{i+1},c_1,t_{i+1}\} \cdots \{b_k,c_{k-i},t_k\}}
    \{b_{k+1},t_{k+1}\} \cdots \{b_m,t_m\}        \sr{R1}{\ra}$

$~$ \hspace{.15in}    
$\{b_1,t_1\} \cdots \{b_{i},t_{i}\}\{b_{i+1},d_1,t_{i+1}\} \cdots \{b_k,d_{k-i},t_k\} 
    \{b_{k+1},t_{k+1}\} \cdots \{b_m,t_m\}$

and

$\{a_1,t_1\} \cdots \{a_{i},t_{i}\}\uu{\{a_{i+1},c_1,t_{i+1}\} \cdots \{a_k,c_{k-i},t_k\}} 
    \{a_{k+1},t_{k+1}\} \cdots \{a_m,t_m\}        \sr{R1}{\ra}$
    
$~$ \hspace{.15in}    
$\uu{\{a_1,t_1\} \cdots \{a_{i},t_{i}\}\{a_{i+1},d_1,t_{i+1}\} \cdots \{a_k,d_{k-i},t_k\} 
    \{a_{k+1},t_{k+1}\} \cdots \{a_m,t_m\}}        \sr{R1}{\ra}$

$~$ \hspace{.15in}    
$\{b_1,t_1\} \cdots \{b_{i},t_{i}\}\{b_{i+1},d_1,t_{i+1}\} \cdots \{b_k,d_{k-i},t_k\} 
    \{b_{k+1},t_{k+1}\} \cdots \{b_m,t_m\}$.

\bigskip

\noindent {\bf 11int option b:}

\noindent For critical pair $a_1 \cdots a_m \ra_i b_1 \cdots b_k$
and $a_p \cdots a_q \ra_i c_p \cdots c_r$ of $S_i$ (where 
$1 \le p \le q \le m$,  $k \le m$, and $r \le q$):
There are rewritings $b_1 \cdots b_k \sr{*}{\ra_i} d_1 \cdots d_s$
and $a_1 \cdots a_{p-1} c_p \dots c_r a_{q+1} \cdots a_m \sr{*}{\ra_i} d_1 \cdots d_s$.  Then:

\bigskip

\noindent If $t_l \neq \emptyset$ for all $r+1 \le l \le q$:

$\uu{\{a_1,t_1\} \cdots \{a_m, t_m\}}   \sr{R1}{\ra}
\uu{\{b_1,t_1\} \cdots \{b_k,t_k\}} t_{k+1} \cdots t_{m} 
   \sr{R1*}{\ra}$

$~$ \hspace{.15in}
$\{d_1,t_1\} \cdots \{d_s,t_s\} t_{s+1} \cdots  t_m$

and

$\{a_1,t_1\} \cdots \{a_{p-1}, t_{p-1}\} \uu{\{a_p, t_p\} \cdots 
       \{a_q, t_q\}} \{a_{q+1}, t_{q+1}\} \cdots \{a_m, t_m\}   \sr{R1}{\ra}$
    
$~$ \hspace{.15in}    
$\{a_1,t_1\} \cdots \{a_{p-1},t_{p-1}\} \{c_p,t_p\} \cdots \{c_r,t_r\} 
    \uu{t_{r+1} \cdots t_q \{a_{q+1},t_{q+1}\} \cdots \{a_m, t_m\}}
    \sr{R2*}{\ra}$
    
$~$ \hspace{.15in}    
$\uu{\{a_1,t_1\} \cdots \{a_{p-1},t_{p-1}\} \{c_p,t_p\} \cdots \{c_r,t_r\} 
    \{a_{q+1},t_{r+1}\} \cdots \{a_{m},t_{m-(q-r)}\}} t_{m-(q-r)+1} \cdots t_m   
    \sr{R1*}{\ra}$

$~$ \hspace{.15in}
$\{d_1,t_1\} \cdots \{d_s,t_s\} t_{s+1} \cdots  t_q$.  

\bigskip

\noindent If $t_{r+1} \cdots t_q = t_{r+1}' \cdots t_u'$ for some 
$u \le q-1$ with each $t_l' \in T$:

\noindent Let $t_l'':=t_l$ for $1 \le l \le r$, $t_{l}'':=t_{l}'$
for $r+1 \le l \le u$, and $t_l'':=t_{l+q-u}$ for $u+1 \le l \le m-(q-u)$.

$\uu{\{a_1,t_1\} \cdots \{a_m, t_m\}}   \sr{R1}{\ra}$
   
$~$ \hspace{.15in}   
$\uu{\{b_1,t_1\} \cdots \{b_k,t_k\}} t_{k+1} \cdots t_{m}  \sr{R1*}{\ra}$
   
$~$ \hspace{.15in}
$\uu{\{d_1,t_1\} \cdots \{d_s,t_s\} t_{s+1} \cdots  t_m} \sr{R2*}{\ra}$

$~$ \hspace{.15in}
$\{d_1,t_1''\} \cdots \{d_s,t_s''\} t_{s+1}'' \cdots  t_{m-(q-u)}''$

and

$\{a_1,t_1\} \cdots \{a_{p-1}, t_{p-1}\} \uu{\{a_p, t_p\} \cdots 
       \{a_q, t_q\}} \{a_{q+1}, t_{q+1}\} \cdots \{a_q, t_q\}   \sr{R1}{\ra}$

$~$ \hspace{.15in}
$\{a_1,t_1\} \cdots \{a_{p-1},t_{p-1}\} \{c_p,t_p\} \cdots \{c_r,t_r\} 
    t_{r+1} \cdots t_q \{a_{q+1},t_{q+1}\} \cdots \{a_m, t_m\} =$

$~$ \hspace{.15in}
$\{a_1,t_1''\} \cdots \{a_{p-1},t_{p-1}''\} \{c_p,t_p''\} \cdots \{c_r,t_r''\} 
    \uu{t_{r+1}'' \cdots t_u'' \{a_{q+1},t_{u+1}''\} \cdots \{a_m, t_{m-(q-u)}''\}} 
\sr{R2*}{\ra}$

$~$ \hspace{.15in}
$\uu{\{a_1,t_1''\} \cdots \{a_{p-1},t_{p-1}''\} \{c_p,t_p''\} \cdots \{c_r,t_r''\} 
     \{a_{q+1},t_{r+1}''\} \cdots \{a_m, t_{m-(q-r)}''\}} t_{m-(q-r)+1}'' \cdots t_{m-(q-u)}''
\sr{R1*}{\ra}$

$~$ \hspace{.15in}
$\{d_1,t_1''\} \cdots \{d_s,t_s''\} t_{s+1}'' \cdots  t_{q-(m-u)}''$.  

\bigskip

Thus every critical pair of the rewriting system $(T,R)$ is resolved,
and so this system is confluent.  This concludes the proof that
this system is complete.
\end{proof}

This also concludes the proof of Theorem~\ref{thm:kb}.


\begin{thebibliography}{1}

\bibitem{brazil}  Brazil, M.,
{\em Calculating growth functions for groups using automata},
Computational algebra and number theory (Sydney, 1992), 1--18,
Math.~Appl.~{\bf 325}, Kluwer Acad.~Publ., Dordrecht, 1995. 

\bibitem{bdc}  Breuillard, E.~and de Cornulier, Y.,
{\em On conjugacy growth for solvable groups},
Illinois J.~Math.~{\bf 54} (2010), 389--395. 

\bibitem{chiswell}  Chiswell, I.M.,
{\em The growth series of a graph product},
Bull.~London Math.~Soc.~{\bf 26} (1994),  268--272. 

\bibitem{conway} Conway, J.B.,
Functions of one complex variable,
Second edition, Graduate Texts in Mathematics
{\bf 11}, Springer-Verlag, New York-Berlin, 1978.


\bibitem{ck} Coornaert, M.~and Knieper, G.,
{\em An upper bound for the growth of conjugacy classes 
in torsion-free word hyperbolic groups},
Internat.~J.~Algebra Comput.~{\bf 14} (2004), 395--401. 

\bibitem{dlh}  De la Harpe, P.,
Topics in geometric group theory,
Chicago Lectures in Mathematics, University of Chicago Press, 
Chicago, IL, 2000.


\bibitem{echlpt}
Epstein, D.B.A., Cannon, J., Holt, D., Levy, S., Paterson, M., and Thurston, W.,
{Word Processing in Groups},
Jones and Bartlett, Boston, 1992.


\bibitem{gn}  Grigorchuk, R.~and Nagnibeda, T.,
{\em Complete growth functions of hyperbolic groups},
Invent.~Math.~{\bf 130} (1997),  159--188. 

\bibitem{gubasapir} Guba, V.~and Sapir, M.,
{\em On the conjugacy growth functions of groups},
Illinois J.~Math.~{\bf 54}  (2010),  301--313. 

\bibitem{hm}
Hermiller, S.~and Meier, J.,
{\em Algorithms and geometry for graph products of groups},
J.~Algebra {\bf 171} (1995), 230--257.

\bibitem{HRR}
Holt, D.F., and Rees, S. and R\"over, C.,
{\em Groups with Context-Free Conjugacy Problems},
Int. J. Alg. Comput.~{\bf 21} (2011), 193--216.

\bibitem{hu}
Hopcroft, J.~and Ullman, J.D.,
Introduction to automata theory, languages, and computation,
Addison-Wesley Series in Computer Science, 
Addison-Wesley Publishing Co., Reading, Mass., 1979.

\bibitem{hullosin}
Hull, M.~and Osin, ~D.,
{\em Conjugacy growth of finitely generated groups},
arXiv:1107.1826v2.


\bibitem{lmw}
Loeffler, J., Meier, J., and Worthington, J.,
{\em Graph products and Cannon pairs},
Internat.~J.~Algebra Comput.~{\bf 12} (2002),  747--754.

\bibitem{lyndonschupp} Lyndon, R.C.~and Schupp, P.E.,
Combinatorial group theory,
Reprint of the 1977 edition, Classics in Mathematics,
Springer-Verlag, Berlin, 2001. 

\bibitem{mann}
Mann, A.,
How groups grow,
London Mathematical Society Lecture Note Series {\bf 395}, 
Cambridge University Press, Cambridge, 2012.


\bibitem{neushap}  Neumann, W.~and Shapiro, M.,
{\em Automatic structures, rational growth, and 
geometrically finite hyperbolic groups},
Invent.~Math.~{\bf 120} (1995),  259--287. 

\bibitem{rivin}
  Rivin, I.,
 \newblock{\em Growth in free groups (and other stories) - twelve years later,}
 \newblock{Illinois J.~Math.} {\bf 54} (2010), 327--370.

\bibitem{sims}
Sims, C.C.,
Computation with finitely presented groups,
Encyclopedia of Mathematics and its Applications {\bf 48}, 
Cambridge University Press, Cambridge, 1994.

\bibitem{stoll}
Stoll, M.,
{\em Rational and transcendental growth series for 
the higher Heisenberg groups},
Invent.~Math.~{\bf 126} (1996),  85--109. 



\end{thebibliography}
\end{document}